\documentclass[preprint,12pt,sort&compress]{elsarticle}

\usepackage[utf8]{inputenc}
\usepackage{a4,amssymb,amsmath,amsthm}

\usepackage{enumitem}
\usepackage{float}

\usepackage[colorlinks=true]{hyperref}

\theoremstyle{plain}
\newtheorem{theorem}{Theorem}[section]
\newtheorem{lemma}[theorem]{Lemma}
\newtheorem{proposition}[theorem]{Proposition}

\newtheorem{remark}[theorem]{Remark}

\theoremstyle{definition}

\numberwithin{equation}{section}

\def\dhp#1{\mathop {#1}\limits_{+h}}
\def\dhm#1{\mathop{#1}\limits_{-h}}
\def\dtp#1{\mathop {#1}\limits_{+\tau}}
\def\dtm#1{\mathop {#1}\limits_{-\tau}}

\newcommand{\MongeAmp}{\mathcal{M}}

\makeatother

\begin{document}
\begin{frontmatter}



\title{Polynomial Difference Schemes with Energy and Momentum Multipliers on Orthogonal Uniform Meshes: Stencil Restrictions and Monge--Amp\`ere-Type Examples}

\author[SUTaddress]{E.~I.~Kaptsov\corref{correspondingauthor}}

\ead{evgkaptsov@math.sut.ac.th}
\cortext[correspondingauthor]{Corresponding author}

\address[SUTaddress]{School of Mathematical Sciences and Geoinformatics, Institute of Science, \\
Suranaree University of Technology, Nakhon Ratchasima, 30000, Thailand}

\date{\today}

\begin{abstract}
Experience with invariant finite-difference schemes for one-dimensional
hydro\-dynamic-type equations for a function $u(t,x)$ suggests that nonlinear
polynomial schemes on orthogonal uniform meshes with the 9-point stencil cannot
simultaneously preserve difference analogues of the energy and momentum
conservation laws, with the multipliers given by the standard difference
approximations of $u_t$ and $u_x$. For the 9-point stencil this hypothesis is
proved for polynomials of arbitrary degree. On larger stencils it is
investigated by means of the difference analogue of the direct method: for
polynomials of bounded degree, such schemes are found to exist only on
sufficiently large stencils, and the second-order equations they approximate
are of Monge--Amp\`ere type. As a by-product, a scheme for the homogeneous
Monge--Amp\`ere equation is obtained that has five conservation laws, admits
most of the symmetries of the differential equation, and is invariant under the
discrete dihedral group $D_4$. Numerical tests on exact solutions confirm its
second order of approximation and show that the conservation laws are
satisfied on these solutions with the same order of accuracy.
\end{abstract}

\begin{keyword}
finite-difference schemes \sep
polynomial schemes \sep
direct method \sep
conservation-law multipliers \sep
Monge--Amp\`ere equation
\end{keyword}

\end{frontmatter}

\section{Introduction}
\label{sec:Intro}

Over the past years, a number of finite-difference schemes for one-dimensional hydrodynamic-type equations have been constructed within the research group led by V.A.~Dorodnitsyn, including schemes for various shallow water models in Lagrangian coordinates~\cite{DorodnLagrangianShallowWater2020,DorodnShallowWater2021,DorodnGreenNaghdi2020,KaptsovModifiedShallowWater2022,KaptsovInvariantSMHD2023,CheviakovInvariantWaves2020}. Particular attention in these studies has been devoted to invariant conservative schemes. We will return to these schemes in Section~\ref{sec:DirectMethod}, as they provide the main foundation for the present study.

The construction of invariant schemes is based on the symmetry properties of the underlying differential equations, which are systematically studied within the framework of group analysis~\cite{OvsiannikovGroupAnalysis1982,OlverLieGroups1993,BlumanAncoBook2002}. In the context of finite-difference equations, this approach leads to invariant difference schemes that preserve the symmetry properties of the corresponding differential models~\cite{DorodnBook2011,LeviWintYamilovBook2023}. Such schemes may also possess difference analogues of the conservation laws of the original models, i.e., they may be conservative schemes. The latter, however, does not hold for arbitrary invariant schemes:~an invariant scheme constructed without additional considerations often possesses only some of the conservation laws, or may fail to preserve any of them (see discussion in~\cite{KaptsovInvariantSMHD2023}).
Additional methods are therefore required for constructing \emph{conservative} invariant schemes.
Such constructions are usually based on different approaches to deriving conservation laws~\cite{OlverLieGroups1993,BlumanAncoBook2002,NazReviewCLs2008}.
Conservative schemes are also constructed without reference to symmetries, for example by discrete-gradient and discrete variational derivative methods~\cite{Gonzalez1996,FuruihataMatsuo2011}.

A classical approach is provided by Noether's theorem~\cite{Noether1918,IbragimovNoetherExt1969}, which makes it possible, given the variational symmetries and the Lagrangian, to derive conservation laws of the corresponding Euler--Lagrange equations. In the finite-difference case, this approach is sufficiently well developed for ordinary difference schemes~\cite{DorodnDiscreteLagrangians2004,DorodnBook2011}, although its application generally requires additional nontrivial work to identify a suitable invariant mesh equation~\cite{DorodnDiscreteLagrangians2004,DorodnDiscrODE2013,DorodnErmakov2016}. Several approaches to discrete analogues of Noether's theorem and to the application of Lagrangians to finite-difference equations have been developed~\cite{DorodnitsynFiniteDifferenceNoether1993,DorodnitsynNoetherDifference2001,DorodnBook2011,HydonVariationalComplex2004,MansfieldNoetherDifference2019,WhiteMovingFramesDifference2024,MarsdenWest2001}. However, the application of these approaches to partial difference equations appears to remain relatively limited.

Another approach to the construction of conservation laws is provided by the adjoint equation method~\cite{BlumanAncoBook2002} and its difference analogue~\cite{WinternitzAdjoint2015}\footnote{See also references therein.}. In the difference case, its main difficulty lies in finding solutions of the difference adjoint equation, which can be a nontrivial task even for relatively simple schemes. Moreover, for schemes based on partial differential equations, this approach remains largely undeveloped.

In this respect, the difference analogue of the direct method~\cite{AncoBluman1997,BlumanAncoBook2002} is currently more productive.
In principle, the direct method makes it possible to find all local conservation laws of a given system of equations by solving the determining equations obtained from the fact that the variational derivative (the Euler operator) of a divergence expression vanishes identically.
The difference analogue of this approach was apparently first used in~\cite{CheviakovInvariantWaves2020,DorodnLagrangianShallowWater2020}.
It is also worth noting that a closely related approach had been considered earlier in~\cite{GrantMultipleConservation2015} in the context of the KdV equation.
The former approach~\cite{CheviakovInvariantWaves2020,DorodnLagrangianShallowWater2020} made it possible to construct invariant finite-difference schemes preserving conservation laws for one-dimensional linear and nonlinear wave equations~\cite{CheviakovInvariantWaves2020} and shallow water equations in~Lagrangian and~Eulerian coordinates~\cite{DorodnLagrangianShallowWater2020,DorodnShallowWater2021}.

\medskip

One of the significant limitations of the direct method for difference equations is the necessity of using orthogonal meshes in the construction of such schemes. This is due to the fact that the difference analogue of the Euler operator~\cite{KuperschmidtDiscreteLax1985,DorodnNetSpaces1991,DorodnBook2011} is known only for orthogonal meshes. Moreover, following~\cite{DorodnLagrangianShallowWater2020,DorodnShallowWater2021,DorodnGreenNaghdi2020,KaptsovModifiedShallowWater2022,KaptsovInvariantSMHD2023,CheviakovInvariantWaves2020}, we restrict the analysis to \emph{uniform} orthogonal meshes, i.e., meshes with constant difference steps.


\smallskip

In most of the cases mentioned above~\cite{DorodnLagrangianShallowWater2020,DorodnShallowWater2021,DorodnGreenNaghdi2020,KaptsovModifiedShallowWater2022,KaptsovInvariantSMHD2023,CheviakovInvariantWaves2020},
the calculations were performed on the 9-point stencil, which is standard for
constructing schemes for second-order equations. This is because most of the
models considered were written in Lagrangian coordinates, where a
hydrodynamic-type system is usually reduced to a single scalar equation for
the function $u(t,x)$. The order of such an equation is higher than that of
the original first-order Eulerian system and, as a rule, even, so that the
equation can be derived from a Lagrangian. This is one of the reasons why the
Lagrangian description is convenient for constructing conservative schemes, in
addition to being a natural framework for formulating physical
laws~\cite{SedovContMech1994}.

Equations of hydrodynamic type in Lagrangian coordinates possess another remarkable property:~as a rule, their admitted symmetries do not violate the conditions for preserving the orthogonality of the mesh~\cite{DorodnNetSpaces1991,DorodnBook2011}. In general, point transformations corresponding to many symmetries (for example, the Lorentz transformation in relativistic mechanics) distort orthogonal meshes and do not allow invariant schemes to be constructed on such meshes. In the case of hydro\-dynamic-type equations, the admitted Lie algebras usually consist of symmetry generators that preserve orthogonal uniform meshes, thus making these meshes particularly suitable for the construction of invariant schemes. Accordingly, the problem of constructing conservative difference schemes is substantially simplified:~the difference Euler operator and the difference analogue of the direct method can be applied.
A detailed discussion of this and other related approaches to the construction of schemes for equations of hydrodynamic type can be found in~\cite{KaptsovMMIM2024}.

\medskip

The present work is mainly devoted to a phenomenon encountered in the construction of such invariant conservative schemes.
In the course of constructing such schemes for hydrodynamic-type equations~\cite{DorodnLagrangianShallowWater2020,DorodnShallowWater2021,DorodnGreenNaghdi2020,KaptsovModifiedShallowWater2022,KaptsovInvariantSMHD2023}, the following observation was made:
\begin{quote}
\textit{For a function \(u(t,x)\) of two independent variables, \(t\) and \(x\), defined
on a standard 9-point stencil with an orthogonal uniform mesh, it does not
seem possible to construct nonlinear polynomial schemes that simultaneously
preserve the conservation laws of energy and momentum associated with the
generators $\frac{\partial}{\partial{t}}$ and
$\frac{\partial}{\partial{x}}$, respectively.}
\end{quote}
Although some considerations concerning this issue in Lagrangian coordinates
have been given in~\cite{BihloInvariantShallowWater2012}, the general problem of
constructing such schemes remains open.

The question is of interest beyond the specific models considered here. For
wave-type equations and for Lagrangian hydrodynamics, the energy and momentum
conservation laws are among the properties most desirable for a scheme to
inherit, and polynomial schemes on relatively small stencils
are the simplest candidates for this role.
A trade-off between invariants is known in other settings: for Hamiltonian ODEs without additional integrals, a symplectic integrator that conserves energy reproduces the exact flow up to a time reparametrization~\cite{GeMarsden1988}.
If the observation above is correct, polynomial schemes on the 9-point stencil
cannot preserve both laws (energy and momentum) at once, so in constructing such schemes one has to
give up one of them, enlarge the stencil, or turn to a different class of
schemes. The last option includes, in particular, schemes whose construction is aided by specially chosen equations of state or auxiliary difference relations~\cite{Korobitsyn1989,Kozlov2019}; such schemes are not considered here.

The present work investigates how far this restriction extends and
under what conditions it can be circumvented.
The difference analogue of the direct method is used as a systematic search tool,
which also yields new conservative schemes as
by-products, and examples of such schemes are constructed on extended stencils.
Before that, however, we briefly discuss the application of the direct method to
linear and nonlinear equations and describe the problem in more detail.

\bigskip

The remainder of the paper is organized as follows.
Section~\ref{sec:DirectMethod} introduces the main concepts and notation,
discusses the application of the difference analogue of the direct method to
the construction of schemes on orthogonal uniform meshes, and reviews earlier
attempts to construct schemes admitting both the energy and momentum
conservation laws, which led to the hypothesis stated above.
Section~\ref{sec:Classif} contains the main results of the paper. Various stencils and polynomial schemes defined on them are analyzed, leading to a number of propositions. In parallel, several conservative schemes with interesting properties are obtained, including a conservative scheme for the Monge--Amp\`ere equation. Section~\ref{sec:Monge} discusses this scheme in greater detail from the viewpoint of its symmetry properties and verifies its order of approximation and difference conservation laws numerically. Finally, Section~\ref{sec:Conclusion} summarizes the main results and presents concluding remarks.

Technical details are collected in the appendices: \ref{app:Proof3x3} contains the proof that nonlinear polynomial schemes with two centered multipliers do not exist on the $3\times3$ stencil for polynomials of arbitrary degree, and \ref{app:Rational} contains remarks on the application of the direct method to rational schemes.

\section{Direct method and its application to schemes on the 9-point stencil}
\label{sec:DirectMethod}

\subsection{Basic definitions and notation}

First, we recall some standard notions and notation.

For the dependent variable $u(t,x)$, the point symmetry generator of a Lie group is written in the standard form~\cite{OvsiannikovGroupAnalysis1982}:
\begin{equation}\label{GenX}
X = \xi^t \frac{\partial}{\partial t}
    + \xi^x \frac{\partial}{\partial x}
    + \eta \frac{\partial}{\partial u},
\end{equation}
where $\xi^t$, $\xi^x$, and $\eta$ are functions of $t$, $x$, and $u$.

In this work, we abstract, whenever possible, from the physical interpretation of these variables. For definiteness, $t$ may be regarded as time and $x$ as a spatial coordinate, as is natural for most of the equations considered below. When a different physical interpretation is required, as, for example, in the case of Monge--Amp\`ere equations, $t$ and $x$ are treated simply as independent coordinates.
\smallskip

As discussed in the previous section, we restrict our consideration to orthogonal uniform meshes.
A generator preserves an orthogonal uniform mesh if the corresponding group
transformation maps the mesh into a mesh of the same type: the coordinate
lines remain mutually orthogonal, and the distances between neighboring nodes
along each coordinate direction remain constant. In terms of the generator,
this gives the following conditions of mesh uniformity and orthogonality,
respectively~\cite{DorodnNetSpaces1991,DorodnBook2011}:
\begin{equation}\label{mesh_conds_uni}
\dhp{D}\dhm{D}(\xi^x) = 0,
\qquad
\dtp{D}\dtm{D}(\xi^t) = 0,
\end{equation}
\begin{equation}\label{mesh_conds_ortho}
\dhp{D}(\xi^t) = -\dtp{D}(\xi^x).
\end{equation}
Here $\underset{\pm\tau}{D}$ and $\underset{\pm h}{D}$ are finite-difference differentiation operators:
\[
\underset{+\tau}{D} = \frac{\underset{+\tau}{S} - 1}{t_{n+1} - t_n},
\quad
\underset{-\tau}{D} = \frac{1 - \underset{-\tau}{S}}{t_n - t_{n-1}},
\quad
\underset{+h}{D} = \frac{\underset{+h}{S} - 1}{x_{m+1} - x_m},
\quad
\underset{-h}{D} = \frac{1 - \underset{-h}{S}}{x_m - x_{m-1}},
\]
which are defined through the finite-difference shift operators $\underset{\pm\tau}{S}$, $\underset{\pm{h}}{S}$. These operators act on functions of the mesh variables as follows:
\[
\def\arraystretch{1.5}
\begin{array}{c}
\displaystyle
\underset{\pm\tau}{S}
\bigl(f(t_n,x_m,u^n_m)\bigr)
=
f(t_{n\pm1},x_m,u^{n\pm1}_m),
\\
\displaystyle
\underset{\pm h}{S}
\bigl(f(t_n,x_m,u^n_m)\bigr)
=
f(t_n,x_{m\pm1},u^{n}_{m\pm1}).
\end{array}
\]

The indices $n$ and $m$ increase along the $t$- and $x$-axes, respectively. In what follows, we use a 9-point stencil,
\[
\bigl(
t_n,t_{n-1},t_{n+1};
x_m,x_{m-1},x_{m+1};
u^n_m,u^n_{m-1},u^n_{m+1},
u^{n-1}_m,u^{n-1}_{m-1},u^{n-1}_{m+1},
u^{n+1}_m,u^{n+1}_{m-1},u^{n+1}_{m+1}
\bigr),
\]
schematically shown in Figure~\ref{fig:9pt}.

\begin{figure}[ht]
\centering
\includegraphics[width=0.45\linewidth]{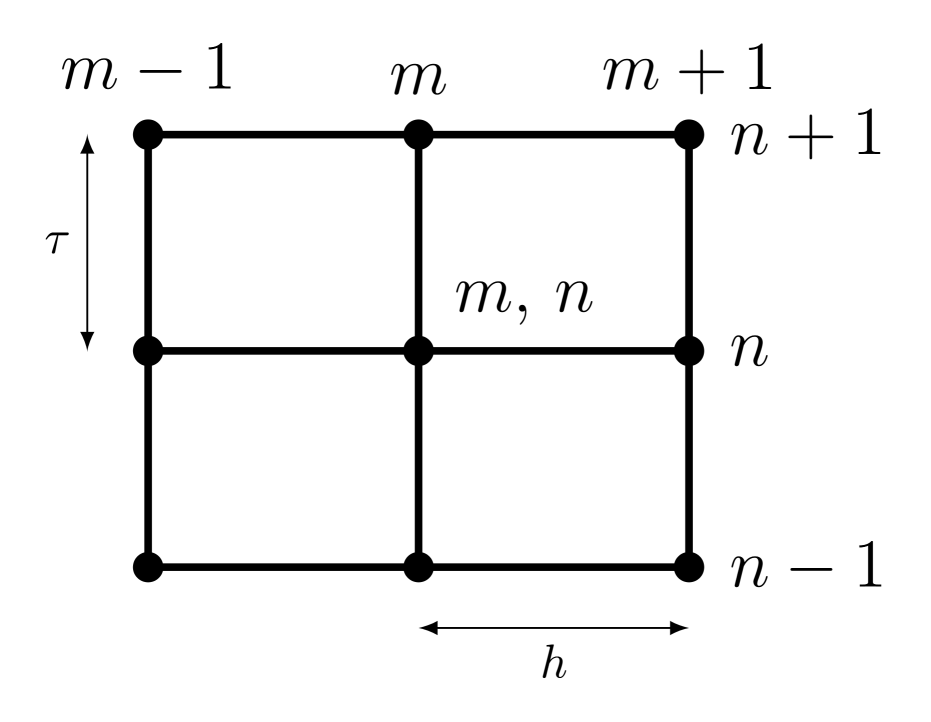}
\caption{9-point stencil for an orthogonal uniform mesh.}
\label{fig:9pt}
\end{figure}

\medskip

The prolongation of the point symmetry generator~(\ref{GenX}) to the finite stencil is obtained by applying the shift operators to its coefficients:
\[
\operatorname{pr} X
=
\sum_{k,l} {\dtp{S}}^{k}
{\dhp{S}}^{l} X
=
\sum_{k,l}
\left[
{\dtp{S}}^{k}{\dhp{S}}^{l}(\xi^t)
\frac{\partial}{\partial t_{n+k}}
+
{\dtp{S}}^{k}{\dhp{S}}^{l}(\xi^x)
\frac{\partial}{\partial x_{m+l}}
+
{\dtp{S}}^{k}{\dhp{S}}^{l}(\eta)
\frac{\partial}{\partial u^{n+k}_{m+l}}
\right],
\]
where the summation extends over all points of the stencil. In particular,
\[
{\dtp{S}}^{k}{\dhp{S}}^{l}(\xi^t)
=
\xi^t(t_{n+k},x_{m+l},u^{n+k}_{m+l}),
\]
and similarly for $\xi^x$ and $\eta$.

A difference equation $\Phi=0$ is invariant under a generator $X$ if the difference prolongation of $X$ satisfies the infinitesimal invariance criterion
\begin{equation}\label{DiffInv}
\left.\operatorname{pr} X(\Phi)\right|_{\Phi=0}=0.
\end{equation}
If the left-hand side vanishes identically, without imposing the equation $\Phi=0$, the function $\Phi$ is said to be \emph{strictly invariant}:
\begin{equation}\label{StrictDiffInv}
\operatorname{pr} X(\Phi)\equiv0.
\end{equation}
For a difference scheme consisting of the difference equation together with the mesh equations, the same invariance criterion is applied to all equations of the scheme.

\medskip

For brevity, throughout this section we sometimes use the Samarskii notation~\cite{SamarskyPopovBook1992}:
\[
\underset{+\tau}{S}(f) = \hat{f},
\qquad
\underset{-\tau}{S}(f) = \check{f},
\qquad
\underset{+h}{S}(f) = f^+,
\qquad
\underset{-h}{S}(f) = f^-,
\]
\[
u_t = \underset{+\tau}{D}(u),
\qquad
\check{u}_{t} = \underset{-\tau}{D}(u),
\qquad
\check{u}_{tt} = \underset{+\tau}{D}\underset{-\tau}{D}(u),
\]
\[
u_x = \underset{+h}{D}(u),
\qquad
u_{\bar x} = \underset{-h}{D}(u),
\qquad
u_{x\bar x} = \underset{+h}{D}\underset{-h}{D}(u),
\]
and so on, whenever it is clear from the context that these symbols denote finite-difference expressions rather than partial derivatives.

\medskip

For a uniform mesh, the steps $\tau$ and $h$ are positive constants:
\begin{equation}\label{UniMesh}
t_{n+k+1} - t_{n+k} = \tau = \mathrm{const},
\qquad
x_{m+l+1} - x_{m+l} = h = \mathrm{const},
\qquad
k,l\in\mathbb{Z}.
\end{equation}
Accordingly, it is convenient to absorb the constant mesh steps into the components of a discrete divergence. In what follows, a difference conservation law is written in the form
\begin{equation}\label{DiffCLGenForm}
\bigl(F\bigr)_m^{n+1}-\bigl(F\bigr)_m^n
+
\bigl(G\bigr)_{m+1}^n-\bigl(G\bigr)_m^n,
\end{equation}
where $F$ is the conserved density and $G$ is the flux in the $x$-direction. We say that this expression is a \emph{difference conservation law} of a difference equation if it vanishes on all solutions of the equation. Both $F$ and $G$ are assumed to be functions of the mesh variables defined on a finite stencil.

\begin{remark}
This form is equivalent to the more conventional divergence expression
\[
\underset{+\tau}{D}(\widetilde F)
+
\underset{+h}{D}(\widetilde G).
\]
Indeed, multiplication by $\tau h$ gives
\[
h\bigl(\widetilde F_m^{n+1}-\widetilde F_m^n\bigr)
+
\tau\bigl(\widetilde G_{m+1}^n-\widetilde G_m^n\bigr),
\]
which has the above form after the rescaling
\[
F=h\widetilde F,
\qquad
G=\tau\widetilde G.
\]
Thus, since the mesh steps are constant, they can be absorbed into the conserved density and flux. This convention simplifies the notation and calculations below. It should be noted, however, that this simplification applies only to the constant factors arising from the finite-difference operators. If a multiplier or another expression depends explicitly on $t_n$ or $x_m$, the dependence on the mesh steps must still be taken into account through~(\ref{UniMesh}).
\end{remark}

\medskip

A finite-difference scheme in its complete form includes not only the approximation of the differential equation but also explicitly specified mesh equations, usually one for each independent variable~\cite{Bakirova1997,DorodnBook2011}. In what follows, the mesh equations will be omitted, as all of them will have the simplest form~(\ref{UniMesh}). However, in general, an explicit specification of the mesh equations is essential, as all basic operations on a finite-difference scheme, including the verification of its conservation laws, are performed with the mesh equations taken into account.

\medskip

The difference Euler operator~\cite{KuperschmidtDiscreteLax1985,DorodnNetSpaces1991} provides a convenient criterion for recognizing divergence expressions and for constructing conservation-law multipliers. Let $P$ be a polynomial in the mesh variables $u^{n+k}_{m+l}$, whose coefficients may depend explicitly on $t_n$ and $x_m$. The difference Euler operator is defined as
\begin{equation} \label{EulerOpDelta}
\mathcal{E}_u
=
\sum_{k,l}
{\dtm{S}}^{k}
{\dhm{S}}^{l}
\frac{\partial}{\partial u^{n+k}_{m+l}}.
\end{equation}
On an orthogonal uniform mesh, $P$ is a divergence expression if and only if
\[
\mathcal{E}_u P\equiv 0.
\]
Consequently, a function $\Lambda$ is a \emph{conservation law multiplier} for a difference equation $\Phi=0$ if
\begin{equation}\label{DirectMethodDeltaMainIdentity}
\mathcal{E}_u(\Lambda\Phi)\equiv 0.
\end{equation}
Such a multiplier corresponds to a difference conservation law of the form~(\ref{DiffCLGenForm}), i.e., $\Lambda\Phi$ is a total difference divergence.

\subsection{Incompatibility of energy and momentum conservation laws for nonlinear schemes}

For finite-difference schemes on orthogonal uniform meshes, the direct method reduces to the application of the  identity~(\ref{DirectMethodDeltaMainIdentity}). It is important to note that this identity holds in this form precisely when the difference steps are assumed to be constant. A significant difference between the differential and difference cases is the following. In the differential case, the differential equation is usually given, and only the conservation-law multipliers are to be determined. In the difference case, however, the scheme itself is often unknown from the outset. Thus, two unknown functions have to be determined simultaneously: the approximation of the differential equation and the multiplier of a conservation law (or, more generally, a linear combination of all possible multipliers).

Since this problem is usually too difficult to solve in full generality, suitable ansatzes have to be introduced. The simplest choices are polynomial and rational functions~\cite{DorodnLagrangianShallowWater2020,DorodnShallowWater2021,DorodnGreenNaghdi2020,KaptsovModifiedShallowWater2022,KaptsovInvariantSMHD2023,CheviakovInvariantWaves2020,KaptsovMMIM2024}. Examples involving more complicated ansatzes can be found, for instance, in~\cite{DorodnShallowWater2021}.

\medskip

In~\cite{CheviakovInvariantWaves2020}, conservative finite-difference schemes were constructed by means of the direct method for the linear wave equation
\begin{equation}\label{LinWaveEq}
u_{tt} - u_{xx} = 0
\end{equation}
and for a nonlinear wave equation of the special form
\begin{equation}\label{NonLinWaveEq}
u_{tt} - (1 + u_{x}^2) u_{xx} = 0.
\end{equation}
Both equations admit the generators
\[
X_1 = \frac{\partial}{\partial{t}},
\qquad
X_2 = \frac{\partial}{\partial{x}},
\]
which, according to Noether's theorem, correspond to the conservation laws of energy and momentum, respectively. The corresponding conservation-law multipliers are $u_t$ and $u_x$.

For the construction of finite-difference schemes, natural approximations to these multipliers, up to constant factors, are
\begin{equation}\label{Lam1Lam2}
\Lambda_1 = u^{n+1}_m-u^{n-1}_m,
\qquad
\Lambda_2 = u^{n}_{m+1}-u^{n}_{m-1}.
\end{equation}

It is straightforward to verify that these are conservation-law multipliers for the simplest linear scheme on the 9-point stencil,
\[
\check{u}_{tt} - u_{x\bar{x}} = 0,
\]
for which the corresponding difference analogues of the energy and momentum conservation laws were constructed~\cite{CheviakovInvariantWaves2020}.

For the nonlinear equation~(\ref{NonLinWaveEq}), the situation becomes considerably more complicated. No reasonable polynomial ansatz on the 9-point stencil was found that would yield a scheme possessing both multipliers~(\ref{Lam1Lam2}) simultaneously. In~\cite{CheviakovInvariantWaves2020}, a classification of polynomial schemes according to the number of conservation laws was carried out, but no class of schemes possessing both the energy and momentum conservation laws, and hence admitting both multipliers, was obtained.

The simplest scheme for~(\ref{NonLinWaveEq}) possessing the largest number of conservation laws was found to be
\[
\check{u}_{tt} - u_{x\bar{x}} - \frac{1}{6} \dhm{D}\left( u_x^2 \, (\hat{u}_x + \check{u}_x)\right) = 0,
\]
which, however, does not possess the momentum conservation law.

Moreover, alternative approximations to the multipliers $u_t$ and $u_x$ did not lead to fundamentally new results either for equation~(\ref{NonLinWaveEq}) or for the equations considered below. Therefore, in the present work, the analysis is mainly restricted to multipliers of the form~(\ref{Lam1Lam2}).

\medskip

Invariant conservative finite-difference schemes for various one-dimensional shallow-water equations in Lagrangian coordinates were constructed in~\cite{DorodnLagrangianShallowWater2020,DorodnShallowWater2021,DorodnGreenNaghdi2020,KaptsovModifiedShallowWater2022,KaptsovInvariantSMHD2023}. All these models can be written in the form
\begin{equation} \label{SWModelsDiff}
u_{tt} - \frac{u_{xx}}{u_x^3} + \Psi(u_x, u_{xx}, u_{tx}, u_{ttx}) + H(u) = 0,
\end{equation}
where $H(u)$ characterizes the bottom profile and $\Psi$ is a rational function of its arguments. For hyperbolic shallow-water equations, $\Psi$ vanishes identically. It is worth noting that, in all the cases considered, when $H(u)=\textrm{const}$, the last equation can be written in divergence form. This implies that the equation admits the conservation-law multiplier $1$, corresponding to the translation generator
$\frac{\partial}{\partial{u}}$.

Using rational-function ansatzes for~(\ref{SWModelsDiff}), invariant conservative finite-dif\-fer\-ence schemes of the form
\[
\check{u}_{tt}
    + \frac{1}{h} \left(
    (\hat{u}_x \check{u}_x)^{-1}
    - (\hat{u}_{\bar{x}} \check{u}_{\bar{x}})^{-1}
    + \widetilde{\Psi} - \widetilde{\Psi}^-
  \right) + \widetilde{H}(u) = 0,
\]
were constructed, where $\widetilde{H}$ and $\widetilde{\Psi}$ approximate $H$ and $\Psi$, respectively, on the 9-point stencil. When $\widetilde{H}(u)=\textrm{const}$, this scheme, as in the differential case, takes a divergence form and admits translations in $u$ and the conservation-law multiplier~$1$.

As in the case of the scheme for equation~(\ref{NonLinWaveEq}), no ansatz was found that would allow both multipliers~(\ref{Lam1Lam2}) to be admitted simultaneously.

A similar problem arises for polynomial schemes for the shallow-water equations in Eulerian coordinates~\cite{DorodnShallowWater2021}.

These results suggest the observation formulated in Section~\ref{sec:Intro} concerning the apparent impossibility of constructing nonlinear polynomial (and, possibly, rational) schemes on the 9-point stencil that simultaneously preserve the conservation laws of energy and momentum.

\medskip

One clarification concerning the notion of momentum conservation used here is necessary. Throughout this paper, the term \emph{conservation law of momentum} refers exclusively to the conservation law associated with the generator
\[
\frac{\partial}{\partial{x}}
\]
and the conservation-law multiplier $u_x$. In Lagrangian coordinates, however, the conservation law associated with translations in $u$,
\[
\frac{\partial}{\partial{u}},
\]
is also commonly interpreted as a conservation law of momentum. In this sense, all the Lagrangian models considered above possess a momentum conservation law. While this interpretation may be sufficient from the physical point of view, it is not sufficient for the algebraic purpose considered here.

A possible explanation for this distinction was given in~\cite{BihloInvariantShallowWater2012}. For the shallow-water equations in Lagrangian coordinates, the momentum equation itself cannot be expressed solely in terms of differential invariants, but only in combination with the continuity equation. Consequently, its direct approximation in terms of difference invariants is not natural. This observation may provide a possible explanation for the incompatibility found above. However, the same argument does not necessarily extend to more general settings, such as the nonlinear wave equation~(\ref{NonLinWaveEq}) or the shallow-water equations in Eulerian coordinates, and therefore does not by itself explain the phenomenon in the general case.

In principle, the direct method applies to rational schemes as well. However, already on the 9-point stencil, and even more so on extended stencils, the calculations become prohibitively cumbersome to perform by hand, and the use of computer algebra systems is usually necessary. At the time when the schemes~\cite{DorodnLagrangianShallowWater2020,DorodnShallowWater2021,DorodnGreenNaghdi2020,KaptsovModifiedShallowWater2022,KaptsovInvariantSMHD2023,CheviakovInvariantWaves2020} were constructed, the available computational resources imposed substantial limitations, making it impractical to systematically examine all possible cases, particularly for extended stencils larger than $3\times3$. In preparing the present study, it became possible to consider substantially larger stencils and to perform a systematic analysis for several of them. This was achieved for polynomial schemes of relatively low degree, for which the calculations arising from the difference analogue of the direct method can be performed within a reasonable amount of time and memory. For rational functions, in contrast, the number of terms and the orders of the difference monomials appearing in the resulting expressions grow rapidly after applying the difference Euler operator, even when the numerator and denominator have relatively low degrees~\cite{KaptsovMMIM2024}. Therefore, the analysis below is restricted to polynomial ansatzes. Under certain additional assumptions, however, the problem for rational schemes can sometimes be simplified; a few remarks on this are given in~\ref{app:Rational}. They are not used in the subsequent sections and may be skipped.

The following sections focus on polynomial schemes and examine in greater detail the restrictions on schemes admitting the multipliers~(\ref{Lam1Lam2}) on the 9-point and extended stencils.

\section{Stencil size and classes of schemes with two specified difference multipliers}
\label{sec:Classif}

The construction of schemes with two difference multipliers leads to a system of algebraic conditions whose structure depends strongly on the stencil and on the form of the scheme and multipliers. The difference analogue of the direct method provides a systematic way to derive these conditions, while computer algebra makes it possible to investigate them for increasingly broad classes of polynomial schemes. This allows several classes of stencils and schemes to be analyzed and a number of structural restrictions to be identified.

The analysis is carried out for polynomial finite-difference schemes of finite order on orthogonal uniform meshes for partial differential equations with two independent variables, $t$ and $x$. Particular attention is paid to the following classes of stencils:

\begin{itemize}
\item Schemes on the 9-point $3\times3$ stencil.

\item Schemes on extended stencils. As will be shown, nontrivial results can already be obtained on a $5\times5$ stencil.

\item Schemes on narrow stencils, i.e., extended asymmetric stencils such as the $2\times6$ or $3\times7$ stencils.
\end{itemize}

We also impose several additional restrictions on the schemes and their multipliers:

\begin{itemize}
\item
Centered multipliers of the form (\ref{Lam1Lam2}), corresponding to the most natural centered approximations.

\item
Non-centered multipliers approximating the differential multipliers $u_t$ and $u_x$ to first order.

\item
Schemes that additionally admit the generator $\frac{\partial}{\partial u}$ and possess a conservation law with the~multiplier $\Lambda=1$ (a divergence form). This requirement is not imposed in the classification below, so the counts of schemes include those that do not have this property. It is used only as a criterion for singling out schemes of particular interest, such as the scheme~(\ref{MAScheme}) of Section~\ref{sec:Monge}.
\end{itemize}

\bigskip

We apply the direct method to the classification problem for schemes
admitting the difference analogues of the multipliers $u_t$ and $u_x$.
Let the scheme be a single relation
\begin{equation}\label{schemeAnsatz}
\Phi\bigl(u_{m+k}^{n+j}\bigr)=0, \qquad (j,k)\in \Omega,
\end{equation}
where the stencil $\Omega$ is a rectangular block of
$N^t\times N^x$ nodes, and $\Phi$ is a polynomial in the values of $u$ at the nodes of $\Omega$ with
constant coefficients.
Following the continuous case, we fix the \emph{centered} multipliers
\begin{equation}\label{LamCentered}
\Lambda_1=u_m^{n+1}-u_m^{n-1}\approx 2\tau u_t,\qquad
\Lambda_2=u_{m+1}^{n}-u_{m-1}^{n}\approx 2h u_x,
\end{equation}
which are symmetric with respect to the node $(t_n,x_m)$.
The unknowns of the classification problem are the coefficients of $\Phi$.

\medskip

Scheme~(\ref{schemeAnsatz}) admits a multiplier $\Lambda$ if and only if
$\Lambda\Phi$ is a total difference divergence on the mesh, i.e.,
$\Lambda\Phi=(\dhm{S}-1)F+(\dtm{S}-1)G$ for some functions
$F$, $G$ of finitely many nodal values of $u$.
This condition is equivalent to identity~(\ref{DirectMethodDeltaMainIdentity}),
which must hold for arbitrary values of $u$ at all mesh nodes.

We look for schemes admitting both multipliers $\Lambda_1$ and $\Lambda_2$.
The multipliers are fixed, so the unknowns are only the coefficients of the polynomial $\Phi$,
and $\Phi$ enters~(\ref{DirectMethodDeltaMainIdentity}) linearly.
Requiring~(\ref{DirectMethodDeltaMainIdentity}) for $\Lambda=\Lambda_1$ and $\Lambda=\Lambda_2$
therefore gives a system of linear homogeneous equations for these coefficients.
The system is overdetermined: there are many more equations than unknowns.

The system splits into independent smaller systems according to the degree of the scheme.
Indeed, call a polynomial \emph{homogeneous of degree~$s$} if all its monomials have the same total degree~$s$
in the nodal variables $u_{m+k}^{n+j}$; for example, $u_m^{n+1}u_m^{n-1}-(u_{m+1}^{n})^2$ is homogeneous of degree~$2$.
Any polynomial $\Phi$ is a sum $\Phi=\sum_s\Phi_s$ of its homogeneous parts $\Phi_s$.
Both multipliers are linear in $u$ and $\Phi$ has constant coefficients,
so the part of degree~$s$ of the left-hand side of~(\ref{DirectMethodDeltaMainIdentity})
is determined by $\Phi_s$ alone. Since a polynomial vanishes identically only if each
of its homogeneous parts vanishes, every $\Phi_s$ must satisfy~(\ref{DirectMethodDeltaMainIdentity})
by itself. It is therefore sufficient to solve the problem separately for each degree~$s$.

For a fixed stencil and a fixed degree $s$ we proceed as follows.
\begin{enumerate}
\item Write $\Phi_s$ as a sum of all possible monomials of degree $s$ in the nodal variables of the stencil,
each with an unknown coefficient.
\item Substitute $\Phi_s$ into the left-hand side of~(\ref{DirectMethodDeltaMainIdentity}) for $\Lambda=\Lambda_1$ and $\Lambda=\Lambda_2$
and expand. The result is a polynomial in the values of $u$ at the mesh nodes
that is linear in the unknown coefficients. Since the identity must hold for arbitrary values of $u$,
the coefficient at every monomial must vanish. Each such condition is a linear
equation in the unknown coefficients of $\Phi_s$.
\item Solve the resulting system of linear equations and find all its nonzero solutions.
Each nonzero solution is a scheme admitting both multipliers; if the system has only the zero solution,
there is no scheme of degree $s$ on this stencil.
\end{enumerate}

Solutions of degree $s=1$ are linear schemes, which are not of interest here;
a solution of degree $s\geqslant 2$ is a genuinely nonlinear scheme.
Note that a scheme $\Phi=0$ is defined up to a constant factor, so the solutions are considered up to such a factor.

The result depends strongly on the stencil: a larger stencil gives more
unknown coefficients, but also more equations. For a fixed stencil and a fixed
degree, the procedure above is a finite computation, so its outcome is either
a complete list of schemes or a proof that none exist. However, the cost of
the computation grows rapidly with both the stencil size and the degree, so it
is feasible only in a limited range, which is specified in the propositions
below. The results obtained in this way hold within these bounds and do not
exclude nonlinear schemes of higher degree or on larger stencils.
For the $3\times3$ stencil, however, the restriction on the degree is removed
by a separate argument given in~\ref{app:Proof3x3}.
Following this procedure, we first examine the smallest square stencils.
It turns out that already on the $3\times3$ and $4\times4$ stencils the
determining system has no nonzero solutions of degree $s\geqslant 2$
(for the $4\times4$ stencil, in the examined range). The smallest square stencil
on which nonlinear schemes appear is of size $5\times5$.

\subsection{Centered multipliers}

Here, we consider schemes admitting the centered conservation-law multipliers~(\ref{Lam1Lam2}).
We begin with the 9-point stencil.

\begin{proposition}\label{prop:33}
There are no nonlinear polynomial schemes on the $3\times3$ stencil
that admit both multipliers~(\ref{Lam1Lam2}).
\end{proposition}

The proof, valid for polynomials of arbitrary degree
(Theorem~\ref{th:33all} in~\ref{app:Proof3x3}), is rather lengthy and requires some
auxiliary constructions, so it is placed in the Appendix, together with the necessary theory,
for the reader interested in the details; further we simply refer to this result.

\begin{remark}
For degrees not exceeding~$14$, Proposition~\ref{prop:33} was also verified by direct computations
according to the algorithm described above.
\end{remark}

For larger stencils, direct computations carried out according to the algorithm described above lead to the following results.

\begin{proposition}\label{prop:44}
There are no nonlinear polynomial schemes of degree $\leqslant 5$ on the $4\times4$ stencil
that admit both multipliers~(\ref{Lam1Lam2}).
\end{proposition}

The absence of schemes on the $4\times4$ stencil may be related to the use of the centered multipliers~(\ref{Lam1Lam2}) on a non-centered stencil. In particular, the $4\times4$ stencil does not have a central node, whereas both multipliers are naturally centered with respect to the stencil.
For the extended $5\times5$ stencil, it is already possible to construct nonlinear polynomial schemes admitting both multipliers~(\ref{Lam1Lam2}). The following proposition describes the classes of schemes obtained.

\begin{proposition}\label{prop:55}
On the $5\times5$ stencil there exist nonlinear polynomial schemes of degree $d=2,3,4$
admitting both multipliers~(\ref{Lam1Lam2}). For these $d$, the equations approximated by such schemes are the following.
\begin{enumerate}[label=\roman*.]
\item
\emph{Second order.}
These are exactly the equations of Monge--Amp\`ere type,
\[
u^{(d-2)/3}\,\MongeAmp\bigl(u^{(d+1)/3}\bigr)=0,\qquad \MongeAmp(v)=v_{tt}v_{xx}-v_{tx}^2,
\]
together with the equations obtained from it by multiplication by $u_t^pu_x^q$ and a power of $u$,
with a simultaneous change of the exponent in the argument of~$\MongeAmp$:
\[
u^{(k-1)(p+q+1)}\,u_t^{\,p}u_x^{\,q}\,\MongeAmp\bigl(u^{k}\bigr)=0,\qquad
k=\frac{d+1}{p+q+3},\quad p,q\geqslant 0,\ p+q\leqslant d-2 .
\]
There are $1,3,6$ such equations for $d=2,3,4$, and no other equations of order at most two occur.

\item
\emph{Third order.}
For $d=2$ there are none. For $d=3,4$ they are  of the form
\[
3q\,\MongeAmp(u)+(u_tu_{xx}-u_xu_{tx})\,q_t+(u_xu_{tt}-u_tu_{tx})\,q_x = 0.
\]
For $d=3$ there are three equations: $q \in \{ u_{tt},\,u_{tx},\,u_{xx} \}$.
For $d=4$ there are fifteen: $q=fg$ with $f\in\{u_{tt},u_{tx},u_{xx}\}$, $g\in\{u,u_t,u_x\}$
(nine equations), and $q=ff^{\prime}$ with $f,f^{\prime}\in\{u_{tt},u_{tx},u_{xx}\}$ (six equations).

\item
\emph{Higher-order equations.}
In addition to the equations listed above, there are
$3, 56, 491$ linearly independent equations for $d=2,3,4$, respectively.
Their orders are all equal to $4$ for $d=2$, and range from $4$ to $8$ for $d=3,4$.
\end{enumerate}
\end{proposition}

For $d\geqslant 5$ the question remains open.

\smallskip

The results of Proposition~\ref{prop:55} show that no hyperbolic second-order equations arise on the $5\times5$ stencil:~all second-order equations reduce to homogeneous equations of Monge--Amp\`ere type, which are parabolic. The third-order equations obtained are likewise closely related to the Monge--Amp\`ere equation.

As one sees from Proposition~\ref{prop:55}, a number of schemes for equations of Monge--Amp\`ere type have been obtained; only a few of them are presented below as illustrations. The higher-order equations on the $5\times5$ stencil are not considered further, since for the purposes of the present analysis it is sufficient to establish the existence of nonlinear schemes admitting both multipliers~(\ref{Lam1Lam2}).
Higher-order equations arising on narrow stencils are considered later in this section, where their number is sufficiently small to allow them to be listed explicitly.

\medskip

Among the quadratic schemes approximating the Monge--Amp\`ere equation, the following scheme
is of particular interest. It is invariant under the action of the discrete symmetries
of the dihedral group $D_4$ on the stencil, namely the rotations by $90^\circ$, $180^\circ$ and $270^\circ$,
the reflections $n\to-n$ and $m\to-m$ with respect to the temporal and spatial axes,
and the reflections with respect to the two diagonals, $n\leftrightarrow m$ and $n\leftrightarrow-m$.
Moreover, it is written in divergence form.
\begin{equation}\label{MAScheme}
M_m^{n} + K_m^{n-1} - K_m^{n+1} - N_{m-1}^{n} + N_{m+1}^{n} \;=\; 0,
\end{equation}
where
\begin{equation}
\label{MASchemeABC}
\begin{aligned}
M_m^n &= A_m^n B_m^n - \bigl(C_m^n\bigr)^2, \\
K_m^n &= C_m^n\, r_m^n - B_m^n\, p_m^n, \\
N_m^n &= A_m^n\, r_m^n - C_m^n\, p_m^n, \\
A_m^n &= u_m^{n+2} - 2u_m^{n} + u_m^{n-2}, \\
B_m^n &= u_{m+2}^{n} - 2u_m^{n} + u_{m-2}^{n}, \\
C_m^n &= u_{m+1}^{n+1} - u_{m-1}^{n+1} - u_{m+1}^{n-1} + u_{m-1}^{n-1}, \\
p_m^n &= u_m^{n+1} - u_m^{n-1},
\qquad
r_m^n = u_{m+1}^{n} - u_{m-1}^{n}.
\end{aligned}
\end{equation}
The fully expanded form of this scheme is rather cumbersome, so it is written in terms of auxiliary difference quantities.

The scheme approximates the homogeneous Monge--Amp\`ere equation~(\ref{MA}) with second order of accuracy in both $\tau$ and $h$: its left-hand side satisfies
\[
\Phi^n_m=48\,\tau^2h^2\,\bigl(\MongeAmp(u)+O(\tau^2+h^2)\bigr),
\]
where the factor $48\tau^2h^2$ is a normalization. It will be considered in more detail in a subsequent section.

\medskip

To illustrate the case $d>2$, one of the schemes with $d=3$, which has a relatively short form, is:
\begin{multline} \label{MA3Scheme}
u_{m-1}^{n}\Bigl[
   \bigl(u_{m-1}^{n} - u_{m-1}^{n-2}\bigr)\bigl(u_{m+1}^{n-1} - u_{m-1}^{n-1}\bigr)
 - \bigl(u_{m}^{n} - u_{m}^{n-2}\bigr)\bigl(u_{m}^{n-1} - u_{m-2}^{n-1}\bigr)
 \Bigr] \\
{} + u_{m-1}^{n+1}\Bigl[
   \bigl(u_{m}^{n} - u_{m-2}^{n}\bigr)\bigl(u_{m}^{n+2} - u_{m}^{n}\bigr)
 - \bigl(u_{m-1}^{n+1} - u_{m-1}^{n-1}\bigr)\bigl(u_{m+1}^{n+1} - u_{m-1}^{n+1}\bigr)
 \Bigr] \\
{} + u_{m}^{n+1}\Bigl[
   \bigl(u_{m+1}^{n+1} - u_{m+1}^{n-1}\bigr)\bigl(u_{m+2}^{n+1} - u_{m}^{n+1}\bigr)
 - \bigl(u_{m+2}^{n} - u_{m}^{n}\bigr)\bigl(u_{m+1}^{n+2} - u_{m+1}^{n}\bigr)
 \Bigr] = 0.
\end{multline}
The scheme corresponds to the equation
\[
u^{1/3} \cdot \MongeAmp\left(u^{4/3}\right)
=
(3 u u_{xx}+ u_{x}^2) u_{tt} +u_{t}^2 u_{xx} - 2 u_{x} u_{t} u_{tx} - 3 u u_{tx}^2 = 0.
\]

\bigskip

The analysis of narrow stencils and polynomial schemes of degree up to 3 leads to the following result.

\begin{proposition}
\label{prop:narrow}
Consider polynomial schemes of degree $\leqslant 3$ admitting both multipliers~(\ref{Lam1Lam2})
on stencils contained in a $3\times7$ stencil (three time layers, seven nodes in~$x$), the multipliers
being centered at the central node of the stencil. The following holds.

\begin{enumerate}[label=\roman*.]
\item
There is a unique, up to a constant factor, quadratic scheme. It approximates the equation
\begin{equation}\label{Class2.1}
u_x u_{txxx} - u_t u_{xxxx} + \frac{3}{2} \left(u_{xx} u_{txx} - u_{tx} u_{xxx} \right) = 0,
\end{equation}
with the leading term of order $\tau h^4$.

\item
Cubic schemes form a $15$-dimensional linear space. The leading term of a generic scheme of this space,
of order $\tau h^4$, approximates the equation
\begin{multline}\label{Class2.2}
u (u_x u_{txxx} - u_t u_{xxxx})
    + \frac{3}{2} \left( u u_{xx} + u_x^2 \right) u_{txx}
    \\
    - \frac{3}{2} \left(u u_{tx} + u_x u_t\right) u_{xxx}
    + \frac{1}{2} \left(u_x u_{tx} - u_t u_{xx} \right) u_{xx}
= 0.
\end{multline}

\item
The remaining $14$ independent cubic schemes approximate equations of higher order, which do not involve $u$
itself (only $u_x$, $u_t$, $u_{tt}$ and their $x$-derivatives). To write them, introduce for functions
$f,g,h$ of $x$ the trilinear operator
\begin{equation}\label{TOperator}
\begin{aligned}
\mathrm{T}[f,g,h]&=f\,g^{\prime}h^{\prime\prime}+3f\,g^{\prime\prime}h^{\prime}+2f\,g^{\prime\prime\prime}h-f^{\prime}g\,h^{\prime\prime} \\
&\quad+3f^{\prime}g^{\prime\prime}h-3f^{\prime\prime}g\,h^{\prime}-3f^{\prime\prime}g^{\prime}h-2f^{\prime\prime\prime}g\,h,
\end{aligned}
\end{equation}
where the prime denotes the derivative with respect to $x$. These equations are
\begin{align}
& \mathrm{T}[u_t,\,u_x,\,\partial_x^{k} u_x]=0, && k=0,1,2,3, && (\tau h^{5+k}), \label{Class2.3}\\
&\mathrm{T}[u_x,\,u_t,\, \partial_x^{k} u_t]=0, && k=0,1,2,3,4, && (\tau^2 h^{4+k}), \label{Class2.4}\\
&\mathrm{T}[u_x,\,u_t,\, \partial_x^{k} u_{tt}]=0, && k=0,1,2,3,4, && (\tau^3 h^{4+k}), \label{Class2.5}
\end{align}
where $\partial_x$ denotes the differential operator with respect to $x$, and the expression in parentheses gives the order of the leading term of the corresponding scheme. These $14$ equations are not differential consequences of~(\ref{Class2.1}) or~(\ref{Class2.2}).
\end{enumerate}

Nonlinear schemes exist only on the whole $3\times7$ stencil; there are none on any of its proper sub-stencils.
The same is true for the $7\times3$ stencil, up to the interchange of $t$ and $x$.
\end{proposition}

\begin{remark}
In Proposition~\ref{prop:narrow}, the \emph{leading term} of a scheme $\Phi=0$ is the first nonzero term of the expansion of $\Phi$ in powers of $\tau$ and $h$. It has the form $\tau^ah^bF[u]$, and the scheme is said to approximate the equation $F[u]=0$. The exponents $a,b$ only indicate the normalization of $\Phi$ and should not be confused with the order of approximation, which is determined by the subsequent terms. The leading terms are given because a space of schemes does not correspond to a single equation: a generic element of the $15$-dimensional space of cubic schemes approximates~(\ref{Class2.2}), whereas special elements, in which the lower-order terms vanish, approximate the other equations, with other exponents $a,b$. Thus, the leading term shows which scheme corresponds to which equation.
\end{remark}

For convenience, the main results of this section are summarized in
Table~\ref{tab:summary}. The equations approximated by the schemes on the
$5\times5$ stencil and on the $3\times7$ stencil are collected in
Tables~\ref{tab:55} and~\ref{tab:narrow}, respectively.

\begin{table}[H]
\centering
\renewcommand{\arraystretch}{1.2}
\begin{tabular}{llll}
\hline
Result & Stencil & Degree $d$ & Outcome \\
\hline
Prop.~\ref{prop:33}     & $3\times3$ & any (Theorem~\ref{th:33all}) & no nonlinear schemes \\
Prop.~\ref{prop:44}     & $4\times4$ & $\leqslant 5$  & no nonlinear schemes \\
Prop.~\ref{prop:55}     & $5\times5$ & $2,3,4$        & schemes exist (Table~\ref{tab:55}) \\
Prop.~\ref{prop:narrow} & $3\times7$ & $\leqslant 3$  & schemes exist only on the full stencil (Table~\ref{tab:narrow}) \\
\hline
\end{tabular}
\caption{Summary of the classification results for the centered multipliers~(\ref{Lam1Lam2}).}
\label{tab:summary}
\end{table}

\begin{table}[H]
\centering
\renewcommand{\arraystretch}{1.2}
\begin{tabular}{lccc}
\hline
Order of the approximated PDE & $d=2$ & $d=3$ & $d=4$ \\
\hline
$2$ (Monge--Amp\`ere type)     & $1$ & $3$     & $6$ \\
$3$                            & $0$ & $3$     & $15$ \\
$\geqslant 4$                  & $3$ & $56$    & $491$ \\
\hline
\end{tabular}
\caption{Number of linearly independent equations approximated by schemes of degree $d$ on the $5\times5$ stencil (Proposition~\ref{prop:55}).}
\label{tab:55}
\end{table}

\begin{table}[H]
\centering
\renewcommand{\arraystretch}{1.2}
\begin{tabular}{cclcc}
\hline
Degree $d$ & $\dim$ & Equation & Number & Order of PDE \\
\hline
$2$ & $1$  & (\ref{Class2.1})                & $1$             & $4$ \\
\hline
$3$ & $15$ & (\ref{Class2.2})                & $1$             & $4$ \\
    &      & (\ref{Class2.3}) & $4$         & $4$--$6$ \\
    &      & (\ref{Class2.4}) & $5$       & $4$--$7$ \\
    &      & (\ref{Class2.5}) & $5$       & $4$--$8$ \\
\hline
\end{tabular}
\caption{Polynomial schemes of degree $\leqslant 3$ on the $3\times7$ stencil admitting both multipliers~(\ref{Lam1Lam2}) (Proposition~\ref{prop:narrow}). Here $\dim$ is the dimension of the space of schemes of the given degree. There are no nonlinear schemes on proper sub-stencils of $3\times7$.}
\label{tab:narrow}
\end{table}

The equations in Table~\ref{tab:narrow} are of fourth and higher order, and,
apparently, none of them has a known physical interpretation.
Nevertheless, we give examples of schemes approximating some of them.

\begin{itemize}

\item
An example of a scheme approximating equation~(\ref{Class2.1}) is
\begin{align*}
& \bigl(u_{m-2}^{n+1}-u_{m-2}^{n-1}\bigr)\bigl(u_{m}^{n}-u_{m-2}^{n}\bigr)
 - \bigl(u_{m-1}^{n+1}-u_{m-1}^{n-1}\bigr)\bigl(u_{m-1}^{n}-u_{m-3}^{n}\bigr) \\
&{} - \bigl(u_{m-1}^{n+1}-u_{m-1}^{n-1}\bigr)\bigl(u_{m+2}^{n}-u_{m}^{n}\bigr)
 + \bigl(u_{m+1}^{n+1}-u_{m+1}^{n-1}\bigr)\bigl(u_{m}^{n}-u_{m-2}^{n}\bigr) \\
&{} + \bigl(u_{m+1}^{n+1}-u_{m+1}^{n-1}\bigr)\bigl(u_{m+3}^{n}-u_{m+1}^{n}\bigr)
 - \bigl(u_{m+2}^{n+1}-u_{m+2}^{n-1}\bigr)\bigl(u_{m+2}^{n}-u_{m}^{n}\bigr) = 0 .
\end{align*}

\item
An example of a scheme approximating equation~(\ref{Class2.2}) is
\begin{align*}
{}& u_{m-1}^{n}\Bigl[\bigl(u_{m-2}^{n+1}-u_{m-2}^{n-1}\bigr)\bigl(u_{m}^{n}-u_{m-2}^{n}\bigr)
 - \bigl(u_{m-1}^{n+1}-u_{m-1}^{n-1}\bigr)\bigl(u_{m-1}^{n}-u_{m-3}^{n}\bigr)\Bigr] \\[4pt]
&{} - u_{m}^{n}\Bigl[\bigl(u_{m-1}^{n+1}-u_{m-1}^{n-1}\bigr)\bigl(u_{m+2}^{n}-u_{m}^{n}\bigr)
 - \bigl(u_{m+1}^{n+1}-u_{m+1}^{n-1}\bigr)\bigl(u_{m}^{n}-u_{m-2}^{n}\bigr)\Bigr] \\[4pt]
&{} + u_{m+1}^{n}\Bigl[\bigl(u_{m+1}^{n+1}-u_{m+1}^{n-1}\bigr)\bigl(u_{m+3}^{n}-u_{m+1}^{n}\bigr)
 - \bigl(u_{m+2}^{n+1}-u_{m+2}^{n-1}\bigr)\bigl(u_{m+2}^{n}-u_{m}^{n}\bigr)\Bigr] = 0 .
\end{align*}

\item
As equations~(\ref{Class2.3})--(\ref{Class2.5}) are defined by the same operator, they can be regarded as
being of the same type, and only one scheme is presented, for an equation of class~(\ref{Class2.5}) with $k=4$:
\begin{multline*}
\mathrm{T}[u_x, u_t, u_{ttxxxx}] =
u_x\,u_{tx}\,u_{ttxxxxxx}
+3\,u_x\,u_{txx}\,u_{ttxxxxx}
\\
+2\,u_x\,u_{txxx}\,u_{ttxxxx}
-u_t\,u_{xx}\,u_{ttxxxxxx}
+3\,u_{xx}\,u_{txx}\,u_{ttxxxx}
\\
-3\,u_t\,u_{xxx}\,u_{ttxxxxx}
-3\,u_{tx}\,u_{xxx}\,u_{ttxxxx}
-2\,u_t\,u_{xxxx}\,u_{ttxxxx}=0.
\end{multline*}
This equation is of order~$6$ in $x$ (the highest derivative is $u_{ttxxxxxx}$) and of total order~$8$.
To write the scheme in a compact form, we introduce the following expressions\:
\begin{align*}
&p_{m+l}=u_{m+l}^{n+1}-u_{m+l}^{n-1}, \qquad
 r_{m+k}=u_{m+k+1}^{n}-u_{m+k-1}^{n},\\
&q_{m+j}=u_{m+j}^{n+1}-2u_{m+j}^{n}+u_{m+j}^{n-1},\\
&\delta q_{m+c}=q_{m+c-2}-4q_{m+c-1}+6q_{m+c}-4q_{m+c+1}+q_{m+c+2},\\
&W_{k,l}=r_{m+k}\,p_{m+l}-r_{m+l}\,p_{m+k}.
\end{align*}
The scheme on the $3\times7$ stencil has the form
\[
\delta q_{m-1}\,W_{-1,-2}+\delta q_{m}\,W_{-1,1}+\delta q_{m+1}\,W_{2,1}=0.
\]
\end{itemize}

It is straightforward to verify using the discrete Euler operator that the presented schemes admit the multipliers~(\ref{Lam1Lam2}).

\subsection{Non-centered multipliers}

We do not discuss these results in further detail, as they are generally very similar to those obtained in the previous section.
It is worth noting, however, that if the centered multipliers $\Lambda_1,\Lambda_2$ are replaced by the cell (non-centered) multipliers
\begin{equation}\label{LambdasCell}
\widetilde{\Lambda}_1 = u_{m}^{n+1}-u_{m}^{n}, \qquad
\widetilde{\Lambda}_2 = u_{m+1}^{n}-u_{m}^{n},
\end{equation}
which approximate $u_t$ and $u_x$ on a single cell, the results remain essentially unchanged. The smallest square stencil is reduced in size:
$4\times4$ instead of $5\times5$ for the Monge--Amp\`ere-type schemes, and $2\times6$ and $6\times2$ instead of $3\times7$ and $7\times3$
for the narrow stencils. As before, no nonlinear schemes exist on smaller stencils (in the range of degrees considered), and the leading parts of the schemes
approximate the same equations. The difference is in the schemes themselves and in the dimensions of the corresponding spaces of schemes:
for instance, the space of cubic schemes on the $2\times6$ stencil is $8$-dimensional and contains only a part of the equations
of the families~(\ref{Class2.3}) and~(\ref{Class2.4}) (none of the family~(\ref{Class2.5}), which requires three time layers),
whereas on the $3\times7$ stencil all of them appear, as for the centered multipliers.

\medskip

In this case, a considerably simpler scheme can be constructed for the Monge--Amp\`ere equation than~(\ref{MA3Scheme}):
\begin{multline} \label{MASchemeCell}
\bigl(u_{m}^{n}-u_{m}^{n-1}\bigr)\bigl(u_{m+2}^{n-1}-u_{m+1}^{n-1}\bigr)
 - \bigl(u_{m+1}^{n}-u_{m+1}^{n-1}\bigr)\bigl(u_{m+1}^{n-1}-u_{m}^{n-1}\bigr) \\
{} + \bigl(u_{m+1}^{n+1}-u_{m+1}^{n}\bigr)\bigl(u_{m+1}^{n+1}-u_{m}^{n+1}\bigr)
 - \bigl(u_{m}^{n+2}-u_{m}^{n+1}\bigr)\bigl(u_{m+2}^{n}-u_{m+1}^{n}\bigr) \\
{} + \bigl(u_{m-1}^{n+2}-u_{m-1}^{n+1}\bigr)\bigl(u_{m}^{n}-u_{m-1}^{n}\bigr)
 - \bigl(u_{m-1}^{n+1}-u_{m-1}^{n}\bigr)\bigl(u_{m}^{n+1}-u_{m-1}^{n+1}\bigr) = 0 .
\end{multline}
Unlike~(\ref{MAScheme}), however, the latter scheme does not admit the whole group~$D_4$.
Moreover, it can be verified that there are no schemes of degree $\leqslant 4$ on this stencil for the Monge--Amp\`ere equation
that admit both the multipliers~(\ref{LambdasCell}) and the group~$D_4$.

\section{Further analysis of the ``dihedral'' scheme for the Monge--Amp\`ere equation}
\label{sec:Monge}

\subsection{Symmetry analysis and conservation laws}

The difference analogue of the direct method does not explicitly use the symmetry properties of the underlying differential equation. It is therefore of interest to examine what information about these symmetries is reflected in a finite-difference scheme constructed by this method. The homogeneous Monge--Amp\`ere equation provides a particularly suitable example for such an analysis, since in the differential case it admits a 15-dimensional Lie algebra of point symmetries~\cite{IbragimovHandbookLieVol1}.

As the analysis above has led to the construction of a number of schemes for the
homogeneous Monge--Amp\`ere equation
\begin{equation}\label{MA}
\MongeAmp(u) = u_{tt} u_{xx} - u_{tx}^2 = 0,
\end{equation}
we consider in more detail one of these schemes, namely, the conservative
scheme (\ref{MAScheme}), which also admits the discrete group~$D_4$.
This scheme is not intended as an alternative to numerical methods for the
elliptic Monge--Amp\`ere equation $\det D^2u=f$, such as the monotone and filtered
schemes~\cite{Oberman2008,FroeseOberman2013}, which are designed to converge
to viscosity solutions. Note that equation~(\ref{MA}) corresponds to the degenerate case $f=0$.
Our aim is rather to see what a scheme obtained by the direct method inherits from the symmetries
and conservation laws of the underlying equation. We show that this scheme
possesses five finite-difference analogues of the local conservation laws of
the original equation and admits a substantial part of the Lie algebra of
generators admitted by equation~(\ref{MA}).

\smallskip

Eq. (\ref{MA}) admits the following 15-parameter Lie algebra~\cite{IbragimovHandbookLieVol1}:
\begin{gather*}
X_1 = \frac{\partial}{\partial t}, \qquad
X_2 = \frac{\partial}{\partial x}, \qquad
X_3 = \frac{\partial}{\partial u}, \qquad
X_4 = t\frac{\partial}{\partial t}, \\[6pt]
X_5 = x\frac{\partial}{\partial x}, \qquad
X_6 = u\frac{\partial}{\partial u}, \qquad
X_7 = x\frac{\partial}{\partial t}, \qquad
X_8 = t\frac{\partial}{\partial x}, \\[6pt]
X_9 = t\frac{\partial}{\partial u}, \qquad
X_{10} = x\frac{\partial}{\partial u}, \qquad
X_{11} = u\frac{\partial}{\partial t}, \qquad
X_{12} = u\frac{\partial}{\partial x}, \\[6pt]
X_{13} = t^2\frac{\partial}{\partial t} + tx\frac{\partial}{\partial x} + tu\frac{\partial}{\partial u}, \qquad
X_{14} = tx\frac{\partial}{\partial t} + x^2\frac{\partial}{\partial x} + xu\frac{\partial}{\partial u}, \\[6pt]
X_{15} = tu\frac{\partial}{\partial t} + xu\frac{\partial}{\partial x} + u^2\frac{\partial}{\partial u}.
\end{gather*}
Using the invariance criteria~(\ref{DiffInv}), it follows that scheme~(\ref{MAScheme}) admits the first twelve generators $X_1,\dots,X_{12}$, but not $X_{13}$, $X_{14}$, and $X_{15}$. Moreover, only $X_1,\dots,X_6$, $X_9$, and $X_{10}$ are compatible with the orthogonal uniform mesh, as they satisfy the conditions~\cite{DorodnNetSpaces1991,DorodnBook2011}.
Hence, in what follows we restrict ourselves to
these eight generators and ask which of them give rise to conservation laws.

\medskip

Following~\cite{OlverLieGroups1993}, we first write down their characteristics
(up to an irrelevant overall sign):
\[
Q_1=u_t,\quad Q_2=u_x,\quad Q_3=1,\quad Q_4=t u_t,\quad Q_5=x u_x,
\]
\[
Q_6=u,\quad Q_9=t,\quad Q_{10}=x.
\]
Characteristics of symmetries are natural candidates for conservation-law
multipliers~\cite{OlverLieGroups1993,BlumanCheviakovAncoBook2010}.
However, not every characteristic is a multiplier: by Noether's theorem,
this holds only for variational symmetries. To select the multipliers among
the $Q_i$, we apply the differential Euler operator $\mathsf{E}_u$ to the
products $Q_i\Phi$ and retain those $Q_i$ for which
\[
\mathsf{E}_u\bigl(Q_i\,\Phi\bigr)\equiv 0,\qquad i\in\{1,2,3,4,5,6,9,10\}.
\]
Here, $\Phi=0$ denotes equation~(\ref{MA}), while $\Phi^n_m=0$ will denote scheme~(\ref{MAScheme}).
This calculation shows that only $Q_1$, $Q_2$, $Q_3$, $Q_9$, and $Q_{10}$, corresponding to $X_1$, $X_2$, $X_3$, $X_9$, and $X_{10}$, respectively, are conservation-law multipliers, whereas $Q_4$, $Q_5$, and $Q_6$ are not.

The differential conservation laws of~(\ref{MA}) are known
(see, e.g.,~\cite{SaminaArifJhangeerWali2025,IbragimovHandbookLieVol1}) and are omitted here;
their discrete counterparts are derived below.

\medskip

The difference multipliers corresponding to $X_1$, $X_2$, $X_3$ are already known:
\[
\Lambda_1=u^{n+1}_m-u^{n-1}_m,\qquad
\Lambda_2=u^{n}_{m+1}-u^{n}_{m-1},\qquad
\Lambda_3=1.
\]
Two more multipliers, corresponding to $X_9$ and $X_{10}$, are readily obtained:
\[
\Lambda_9=t_n,\qquad \Lambda_{10}=x_m.
\]

For each $i\in\{1,2,3,9,10\}$ we found fluxes $F^{(i)}$, $G^{(i)}$ (polynomials in $u$ on the
stencil, with coefficients depending on $t_n$, $x_m$ for $i=9,10$) such that
\begin{equation}\label{DiffCL}
\Lambda_i\,\Phi_m^n=\bigl(F^{(i)}\bigr){}_m^{n+1}-\bigl(F^{(i)}\bigr){}_m^{n}
+\bigl(G^{(i)}\bigr){}_{m+1}^{n}-\bigl(G^{(i)}\bigr){}_m^{n},
\end{equation}
so that on solutions of the scheme, $\Lambda_i\Phi^n_m=0$ is a difference conservation law.

For each multiplier $\Lambda_i$, the corresponding fluxes are given using the quantities~(\ref{MASchemeABC}).

\begin{itemize}

\item
$\Lambda_1= u^{n+1}_m-u^{n-1}_m = p_m^n$:
\begin{align*}
F^{(1)}{}_m^n &= \bigl(B_m^{n-1}+B_m^{n}\bigr)p_m^{n-1}p_m^{n}-C_m^{n-1}\,p_m^{n}\,r_m^{n-1}-C_m^{n}\,p_m^{n-1}\,r_m^{n},\\
G^{(1)}{}_m^n &= A_{m-1}^{n}\,p_m^{n}\,r_{m-1}^{n}+A_{m}^{n}\,p_{m-1}^{n}\,r_{m}^{n}-\bigl(C_{m-1}^{n}+C_{m}^{n}\bigr)p_{m-1}^{n}p_{m}^{n}.
\end{align*}

\item
$\Lambda_2= u^{n}_{m+1}-u^{n}_{m-1} = r_m^n$:
\begin{align*}
F^{(2)}{}_m^n &= B_m^{n-1}\,p_m^{n-1}\,r_m^{n}+B_m^{n}\,p_m^{n}\,r_m^{n-1}-\bigl(C_m^{n-1}+C_m^{n}\bigr)r_m^{n-1}r_m^{n},\\
G^{(2)}{}_m^n &= \bigl(A_{m-1}^{n}+A_{m}^{n}\bigr)r_{m-1}^{n}r_{m}^{n}-C_{m-1}^{n}\,p_{m-1}^{n}\,r_{m}^{n}-C_{m}^{n}\,p_{m}^{n}\,r_{m-1}^{n}.
\end{align*}

\item
$\Lambda_3=1$:
\begin{align*}
F^{(3)}{}_m^n &= B_m^{n-1}p_m^{n-1}+B_m^{n}p_m^{n}-\bigl(C_m^{n-1}+C_m^{n}\bigr)\bigl(r_m^{n-1}+r_m^{n}\bigr),\\
G^{(3)}{}_m^n &= \bigl(A_{m-1}^{n}+A_{m}^{n}\bigr)\bigl(r_{m-1}^{n}+r_{m}^{n}\bigr)-C_{m-1}^{n}p_{m-1}^{n}-C_{m}^{n}p_{m}^{n}.
\end{align*}

\item $\Lambda_9=t_n$:
\begin{align*}
F^{(9)}{}_m^n &= t_n F^{(3)}{}_m^n+\tau\,\widehat F_m^n, \qquad
G^{(9)}{}_m^n = t_n G^{(3)}{}_m^n+\tau\,\widehat G_m^n,\\
\widehat F_m^n &= \bigl(r_m^{n-1}+r_m^{n}\bigr)\bigl(r_m^{n}+r_m^{n+1}\bigr)+r_{m+1}^{n}\bigl(r_{m+1}^{n-1}+r_{m+1}^{n+1}\bigr),
\\
\widehat G_m^n &= -\bigl(p_m^{n}r_{m-1}^{n}+p_{m+1}^{n}r_{m}^{n}+p_m^{n+1}r_{m-1}^{n+1}+p_{m+1}^{n+1}r_{m}^{n+1}\bigr).
\end{align*}

\item $\Lambda_{10}=x_m$:
Here the roles of $n$ and $m$ are exchanged ($A\leftrightarrow B$, $p\leftrightarrow r$, $\tau\to h$); one uses the second (mirror) divergence form of the scheme,
\begin{align*}
F^{(10)}{}_m^n &= x_m F^{(3')}{}_m^n+h\,\widetilde F_m^n, \qquad
G^{(10)}{}_m^n = x_m G^{(3')}{}_m^n+h\,\widetilde G_m^n,\\
\widetilde F_m^n &= -\bigl(r_m^{n}p_m^{n-1}+r_m^{n+1}p_m^{n}+r_{m+1}^{n}p_{m+1}^{n-1}+r_{m+1}^{n+1}p_{m+1}^{n}\bigr), \\
\widetilde G_m^n &= \bigl(p_{m-1}^{n}+p_{m}^{n}\bigr)\bigl(p_{m}^{n}+p_{m+1}^{n}\bigr)+p_{m}^{n+1}\bigl(p_{m-1}^{n+1}+p_{m+1}^{n+1}\bigr),
\end{align*}
where
\begin{align*}
F^{(3')}{}_m^n &= \bigl(B_m^{n-1}+B_m^{n}\bigr)\bigl(p_m^{n-1}+p_m^{n}\bigr)-C_m^{n-1}r_m^{n-1}-C_m^{n}r_m^{n},\\
G^{(3')}{}_m^n &= A_{m-1}^{n}r_{m-1}^{n}+A_{m}^{n}r_{m}^{n}-\bigl(C_{m-1}^{n}+C_{m}^{n}\bigr)\bigl(p_{m-1}^{n}+p_{m}^{n}\bigr).
\end{align*}

\end{itemize}

Thus, the scheme~(\ref{MAScheme}) possesses discrete analogues of the
conservation laws that correspond to all the generators $X_1$, $X_2$, $X_3$,
$X_9$, $X_{10}$ which are compatible with the orthogonal uniform mesh and
give conservation laws of the differential equation~(\ref{MA}). The
generators $X_4$, $X_5$, $X_6$, which are compatible with the mesh but do not
yield conservation laws even in the differential case, give none in the
difference case either. The remaining generators admitted by the scheme
($X_7$, $X_8$, $X_{11}$, $X_{12}$) violate the mesh conditions and fall outside
the scope of the difference Euler operator used here.

\subsection{Numerical verification}
\label{sec:MANumerics}

The properties of the scheme~(\ref{MAScheme}) established above, namely its approximation of equation~(\ref{MA}) and the existence of the difference conservation laws~(\ref{DiffCL}), were obtained analytically; the identities~(\ref{DiffCL}) ($i \in \{1,2,3,9,10\}$) were also verified symbolically with computer algebra. To verify the approximation property independently, and to examine the behaviour of the conservation laws on exact solutions, the expression $\Phi^n_m$ and the fluxes $F^{(i)}$, $G^{(i)}$ were programmed and evaluated on mesh functions, namely on the values of exact solutions of~(\ref{MA}) at the mesh nodes. The scheme itself is not solved here, and no boundary-value problem is considered: the aim is to test it as a discrete model of~(\ref{MA}). Two tests were performed: (i)~the order of approximation of equation~(\ref{MA}) on its exact solutions; (ii)~the balance of the five difference conservation laws on exact solutions.

\medskip

\paragraph*{Test functions}
The exact solutions are taken from the classical description of the solutions of the homogeneous Monge--Amp\`ere equation~(\ref{MA}); see, e.g.,~\cite{Goursat1964CourseIII1,PolyaninZaitsevNPDE2012,IbragimovHandbookLieVol1}.
Its general solution can be written as the envelope of a one-parameter family of planes (i.e., the graph of $u$ is a developable surface),
\begin{equation}\label{MAEnvelope}
u=s\,t+\varphi(s)\,x+\psi(s),\qquad t+\varphi'(s)\,x+\psi'(s)=0,
\end{equation}
with arbitrary functions $\varphi$ and $\psi$; the second relation determines the parameter $s=s(t,x)$. This relation expresses the stationarity of $u$ with respect to $s$, hence $u_t=s$ and $u_x=\varphi(s)$. Therefore $u_x$ is a function of $u_t$, the Hessian of $u$ is degenerate, and $\MongeAmp(u)=0$. Solutions that depend on one arbitrary function of one variable include the functions
\begin{equation}\label{MAexactSolType1}
u=f(at+bx)+ct+dx+e
\end{equation}
and the functions
\begin{equation}\label{MAexactSolType2}
u=t\,g(x/t),
\end{equation}
which are homogeneous of degree one~\cite{PolyaninZaitsevNPDE2012}.

Four test solutions are of the type~(\ref{MAexactSolType2}) with $g(z)=z^2$, $z^3$, $\sqrt{1+z^2}$, $e^{z}$:
\[
u=\frac{x^2}{t},\qquad u=\frac{x^3}{t^2},\qquad u=\sqrt{t^2+x^2},\qquad u=t\,e^{x/t}.
\]
They are considered in the box $1\leqslant t\leqslant2$, $0.5\leqslant x\leqslant1.5$. The fifth solution is a non-degenerate envelope~(\ref{MAEnvelope}) with $\varphi(s)=s^3$, $\psi(s)=2s^2$. In this case the second relation of~(\ref{MAEnvelope}) is the quadratic equation $t+4s+3xs^2=0$, and its root vanishing at $t=0$,
\begin{equation}\label{MAEnvelopeSol}
s=s(t,x)=\frac{-2+\sqrt{4-3xt}}{3x},\qquad u=st+s^3x+2s^2,
\end{equation}
gives an explicit solution, smooth in the box $0\leqslant t\leqslant0.5$, $1\leqslant x\leqslant1.5$ (where $3xt<4$); here $u_t=s$, $u_x=s^3$. This solution is neither homogeneous nor of the form~(\ref{MAexactSolType1}), i.e., it does not belong to either of these special classes. In addition, a smooth function that is \emph{not} a solution, $u=\sin t\cos 2x+tx$ on $[0,1]^2$, is used as a control; for it, $\MongeAmp(u)\neq0$ is known explicitly.
Plane waves $e^{at+bx}$, which are of the form~(\ref{MAexactSolType1}), were deliberately not used: on a uniform mesh, $M^n_m=K^n_m=N^n_m=0$ identically for them, so the scheme is exact on them and the test would be uninformative.

\medskip

\paragraph*{Consistency}
The mesh is uniform with $N$ intervals in each direction, so that $\tau=h=\ell/N$, where $\ell$ is the side of the box. The scheme~(\ref{MAScheme}) is evaluated at all nodes whose $5\times5$ stencil lies inside the mesh. As mentioned above, the quantity $\Phi^n_m/(48\tau^2h^2)$ approximates $\MongeAmp(u)$ with the error $O(\tau^2+h^2)$; hence, for a sufficiently smooth $u$, the maximal error $E_N$ over the nodes is expected to behave as $E_N\approx Ch^2$, and the observed order $p=\log_2(E_{N/2}/E_N)$ is expected to tend to~$2$ as the mesh is refined. The errors are shown in Fig.~\ref{fig:MANum}(a) and Table~\ref{tab:MAcons}. For all test functions $p$ indeed approaches $2$: between $N=160$ and $N=320$ it equals $1.87$--$1.96$ for the exact solutions and $1.99$ for the control function. Thus the numerical results confirm the second order of approximation. For the functions with large derivatives in the box ($x^3/t^2$, $t\,e^{x/t}$) the observed orders are below~$1$ for $N\leqslant20$ and approach~$2$ only on finer meshes, since the higher-order terms of the error are still significant on coarse meshes. For the envelope solution, whose errors are small in absolute value, the convergence is less regular before $N=40$.

\begin{table}[ht]
\def\arraystretch{1.15}
\centering
\small
\begin{tabular}{lccccc}
\hline
$u$ & $E_{40}$ & $E_{80}$ & $E_{160}$ & $E_{320}$ & $p$\\
\hline
$x^2/t$            & $4.9\cdot10^{-3}$ & $1.5\cdot10^{-3}$ & $4.0\cdot10^{-4}$ & $1.0\cdot10^{-4}$ & 1.94\\
$x^3/t^2$          & $9.6\cdot10^{-2}$ & $3.1\cdot10^{-2}$ & $8.9\cdot10^{-3}$ & $2.4\cdot10^{-3}$ & 1.90\\
$\sqrt{t^2+x^2}$   & $1.6\cdot10^{-4}$ & $4.4\cdot10^{-5}$ & $1.2\cdot10^{-5}$ & $3.0\cdot10^{-6}$ & 1.96\\
$t\,e^{x/t}$       & $2.3\cdot10^{-2}$ & $7.8\cdot10^{-3}$ & $2.3\cdot10^{-3}$ & $6.2\cdot10^{-4}$ & 1.89\\
envelope~(\ref{MAEnvelopeSol})           & $4.8\cdot10^{-8}$ & $1.8\cdot10^{-8}$ & $5.2\cdot10^{-9}$ & $1.4\cdot10^{-9}$ & 1.87\\
\hline
$\sin t\cos2x+tx$ (control) & $4.7\cdot10^{-3}$ & $1.2\cdot10^{-3}$ & $3.0\cdot10^{-4}$ & $7.6\cdot10^{-5}$ & 1.99\\
\hline
\end{tabular}
\caption{Maximal consistency error $E_N=\max|\Phi^n_m/(48\tau^2h^2)-\MongeAmp(u)|$ and observed order $p=\log_2(E_{N/2}/E_N)$ at $N=320$ for the solutions of type~(\ref{MAexactSolType2}), the envelope solution~(\ref{MAEnvelopeSol}), and the control function.}
\label{tab:MAcons}
\end{table}

\medskip

\paragraph*{Conservation laws on exact solutions}
Let $R=\{n_0\leqslant n\leqslant n_1,\ m_0\leqslant m\leqslant m_1\}$ be a rectangle of mesh nodes. Summation of~(\ref{DiffCL}) over $R$ gives the identity
\[
\sum_{n=n_0}^{n_1}\sum_{m=m_0}^{m_1}\Lambda_i\,\Phi^n_m
=\sum_{m=m_0}^{m_1}\Bigl(F^{(i)}{}_m^{\,n_1+1}-F^{(i)}{}_m^{\,n_0}\Bigr)
+\sum_{n=n_0}^{n_1}\Bigl(G^{(i)}{}_{m_1+1}^{\,n}-G^{(i)}{}_{m_0}^{\,n}\Bigr),
\]
in which all interior terms have cancelled. If the mesh function is the restriction of an exact solution of~(\ref{MA}), the left-hand side is a sum of local truncation errors of the scheme, and the right-hand side, the net flux through the boundary of $R$, should therefore be small compared with the fluxes themselves. To measure this, the relative imbalance
\[
\delta_i=\frac{\bigl|\sum_{m}(F^{(i)}{}_m^{\,n_1+1}-F^{(i)}{}_m^{\,n_0})+\sum_{n}(G^{(i)}{}_{m_1+1}^{\,n}-G^{(i)}{}_{m_0}^{\,n})\bigr|}
{\sum_{m}(|F^{(i)}{}_m^{\,n_1+1}|+|F^{(i)}{}_m^{\,n_0}|)+\sum_{n}(|G^{(i)}{}_{m_1+1}^{\,n}|+|G^{(i)}{}_{m_0}^{\,n}|)}
\]
was computed on a sequence of meshes with $N=40,80,160,320,640$. The rectangle $R$ was the same in all runs: the central part of the box obtained by discarding $N/10$ nodes at each side. For all five laws and all five exact solutions $\delta_i$ decreases as $N^{-2}$: the observed order is $2.0$ for $N\geqslant160$, with the only exception of $\Lambda_1$ on $u=t\,e^{x/t}$, where the imbalance is an order of magnitude smaller than for the other laws at $N=40$ and approaches the rate $2$ from below ($1.6$, $1.8$, $1.9$). At $N=640$ the imbalance does not exceed $4.4\cdot10^{-6}$ (Fig.~\ref{fig:MANum}(b)). For the control function, for which $\MongeAmp(u)\neq0$, the imbalance does not decrease and stays between $4\cdot10^{-2}$ and $2\cdot10^{-1}$ for all $N$, which shows that the test is sensitive to the violation of the equation. Hence the discrete laws are exact on solutions of the scheme (up to round-off), and on exact solutions of~(\ref{MA}) they are satisfied with second order of accuracy.

\begin{figure}[ht]
\centering
\includegraphics[width=\textwidth]{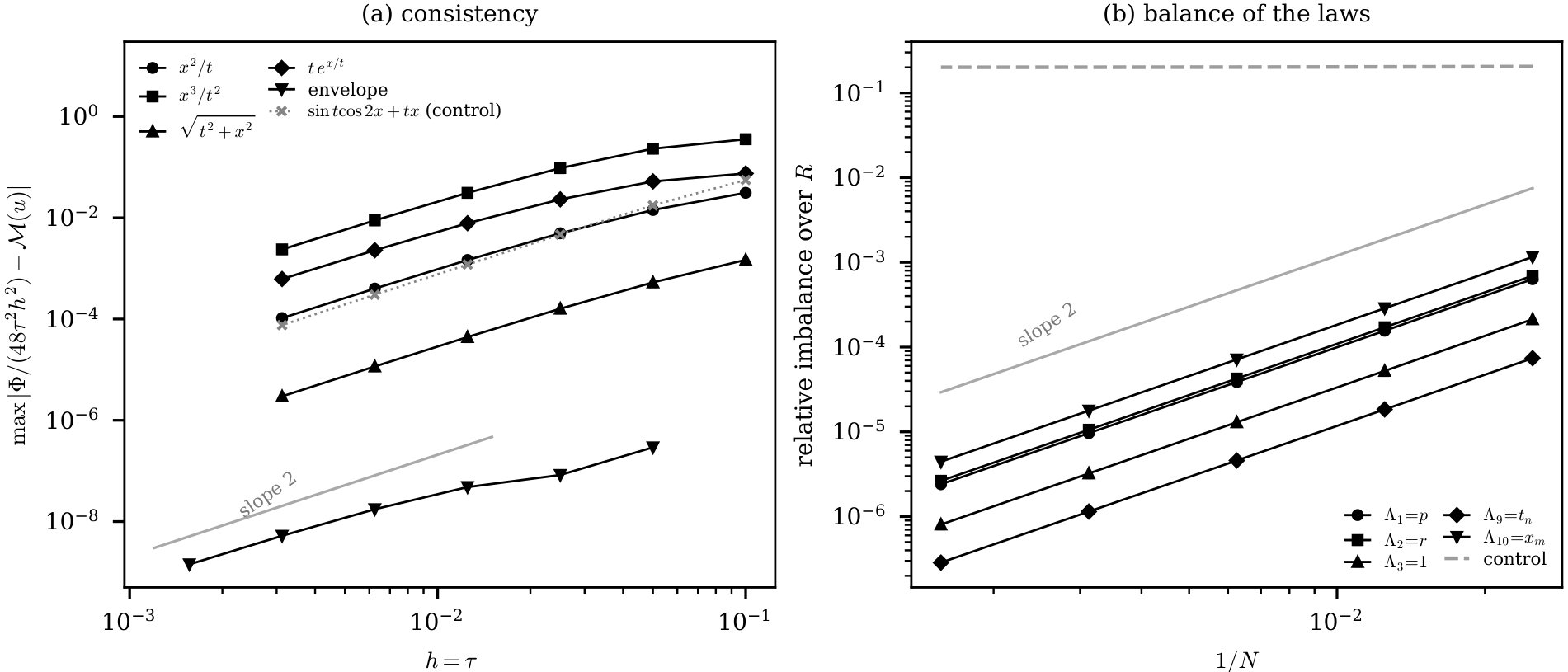}
\caption{Numerical verification of the scheme~(\ref{MAScheme}). (a)~Consistency: maximal error $|\Phi^n_m/(48\tau^2h^2)-\MongeAmp(u)|$ versus $h=\tau$ for the five exact solutions and for the control function (grey, dotted); the grey line has slope~2. (b)~Relative imbalance $\delta_i$ of the five conservation laws over a fixed rectangle $R$ versus $1/N$: for each law, the largest value over the five exact solutions (solid), and the largest value over the five laws for the control function (dashed); the grey line has slope~2.}
\label{fig:MANum}
\end{figure}


\section{Conclusion}
\label{sec:Conclusion}

Based on previous experience with the construction of invariant conservative finite-difference schemes for hydrodynamic-type equations~\cite{DorodnLagrangianShallowWater2020,DorodnShallowWater2021,DorodnGreenNaghdi2020,KaptsovModifiedShallowWater2022,KaptsovInvariantSMHD2023,CheviakovInvariantWaves2020}, a difficulty was observed in constructing nonlinear polynomial and rational schemes on the 9-point stencil that possess difference analogues of both the energy and momentum conservation laws, with conservation-law multipliers approximating $u_t$ and $u_x$, respectively. (For brevity, such schemes will be referred to below as \emph{schemes with two multipliers}.) The constructions mentioned above were mainly carried out using the difference analogue of the direct method on orthogonal uniform meshes.

\medskip

The present work reviews these earlier results and investigates polynomial schemes with two multipliers. A simplification of the direct method for rational schemes, which may be useful in related problems, is given in \ref{app:Rational}: under certain assumptions, the corresponding problem reduces to a linear system with constant coefficients.

\medskip

To determine how essential the restrictions on schemes with two multipliers are,
the $9$-point $3\times3$ stencil was considered first. It was shown that no nonlinear
polynomial schemes of this type exist for a polynomial of any finite degree~(Theorem~\ref{th:33all} in~\ref{app:Proof3x3}).
For degrees not exceeding~$14$ this was also verified by direct computation.

For the extended $4\times4$ stencil, schemes of degree not exceeding~$5$ were considered,
and the same result was obtained.

For the $5\times5$ stencil, however, the situation changes substantially: for polynomial schemes of degree not exceeding 4, a large number of schemes with two multipliers arise. All schemes approximating second-order equations correspond to equations of Monge--Amp\`ere type. Several such schemes were presented, including the ``dihedral'' scheme~(\ref{MAScheme}) admitting the discrete transformations of the group~$D_4$.
Several third-order equations also arise, all involving the Monge--Amp\`ere equation as a component, together with a large number of higher-order equations. These higher-order cases are not considered in detail, as their complete enumeration would require considerable space and, apparently, they do not have an evident physical interpretation.

\smallskip

Narrow stencils with different numbers of nodes in the two independent variables were also considered,
up to the size $3\times7$, including their $7\times3$ counterparts. For these stencils, nonlinear polynomial schemes of degree not exceeding~$3$
exist only on the full $3\times7$ stencil (respectively $7\times3$). They approximate the fourth-order
equations~(\ref{Class2.1}) and~(\ref{Class2.2}), as well as $14$ further equations of orders from~$4$ to~$8$.

\medskip

The case of non-centered multipliers~(\ref{LambdasCell}), defined on a cell, was also considered. The results were generally similar to those obtained for the multipliers~(\ref{Lam1Lam2}). For the Monge--Amp\`ere equations, however, schemes were found on somewhat smaller stencils, although these schemes do not admit the $D_4$ symmetry.

\smallskip

The results obtained are formulated as propositions and summarized in
Tables~\ref{tab:summary}, \ref{tab:55}, and~\ref{tab:narrow}. They do not
provide a complete classification of all possible classes of schemes on all
possible stencils. Nevertheless, they substantially extend our understanding
of the restrictions on schemes with two multipliers. In particular, they
confirm the initial hypothesis that nonlinear polynomial schemes with two
multipliers cannot be constructed on the $3\times3$ stencil: for polynomials
of any finite degree this is proved in~\ref{app:Proof3x3}. For
rational schemes the question remains open.

From the practical point of view, the results indicate that polynomial schemes
on small stencils are unlikely to preserve both the energy and momentum
conservation laws for nonlinear equations of the type considered here. This
refers to schemes given by a single polynomial relation with constant coefficients on an orthogonal
uniform mesh, with the multipliers~(\ref{Lam1Lam2}) or~(\ref{LambdasCell}); the result is rigorous for the
$3\times3$ stencil and polynomials of any degree, and holds for larger stencils within the examined range of degrees.
When constructing conservative schemes, one therefore has to choose between
preserving only one of these laws, using larger stencils, or turning to other classes of
schemes, for example, rational ones.

It also appears that, apart from the linear case, the class of schemes with two
multipliers for hyperbolic equations is rather narrow and is not represented by
polynomials of relatively low degree.

%
%

\medskip

In the final section, a symmetry analysis of the ``dihedral'' scheme~(\ref{MAScheme})
for the Monge--Amp\`ere equation was carried out. The scheme admits 12 of the 15
generators of the Lie algebra of the corresponding differential equation. Some
generators of this algebra violate the conditions~(\ref{mesh_conds_uni}),
(\ref{mesh_conds_ortho}) under which an orthogonal uniform mesh is preserved,
and the difference analogue of the direct method, which is based on such meshes,
does not detect the corresponding symmetries. Consequently, fully invariant
schemes should not in general be expected from this approach when some of the
admitted symmetries violate the mesh conditions. Some of these symmetries are
nevertheless admitted by the scheme, but no corresponding difference
conservation laws appear to exist.

Among the generators compatible with the mesh, five give rise to conservation
laws in both the differential and difference cases. The corresponding
conservation laws and multipliers were obtained explicitly.

Although numerical computations are not the main subject of this work, a numerical check of the scheme~(\ref{MAScheme}) was carried out. Tests on exact solutions of the Monge--Amp\`ere equation, including a solution that is not homogeneous, confirm the second order of approximation of the scheme and show that the five difference conservation laws are satisfied on these solutions with second order of accuracy.

\medskip

For the $3\times3$ stencil, the nonexistence of nonlinear schemes with two multipliers has been reduced to an
algebraic statement about polynomials and proved for arbitrary degree (\ref{app:Proof3x3}),
so that in this case the limitation has a clear algebraic origin.
The same appears to hold for the $4\times4$ stencil, where the problem seems to reduce to the $3\times3$ case
(see Remark~\ref{rem:proof4x4} in~\ref{app:Proof3x3}); a rigorous proof is not given here, and
Proposition~\ref{prop:44} remains established by direct computation for degrees $\leqslant5$.
For larger stencils, the observed behavior, in particular the appearance of a large number of schemes
on the $5\times5$ stencil, is so far established only by direct computation.
Whether it has a similar algebraic explanation, and how the picture changes for polynomials of higher degree
on larger stencils and for rational schemes, appears to be a natural direction for further study.

\section*{Acknowledgments}

The author is grateful to Professors V.A.~Dorodnitsyn and S.V.~Meleshko for valuable discussions and
comments on the problem considered in this paper.

%

%
%

\bigskip
\par
\bigskip


\appendix

\renewcommand{\thetheorem}{\Alph{section}.\arabic{theorem}}
\renewcommand{\theequation}{\Alph{section}.\arabic{equation}}

\section{General proof for polynomials of arbitrary degree on the 9-point stencil}
\label{app:Proof3x3}

For stencils of small size the nonexistence of nonlinear schemes can be proved directly
for polynomials of arbitrary degree. For this it is convenient to leave aside the
difference Euler operator and to use another, equivalent definition of a divergence:
a polynomial is a divergence if its sum over any rectangular region of the mesh
reduces to terms attached to the boundary nodes only, all interior terms cancelling~\cite{GodunovRyabenkii1987}.
The arguments below overlap substantially with the difference variational
complex of~\cite{HydonVariationalComplex2004}. We deliberately avoid its language
and give simpler proofs in terms of monomials and their shifts; in particular,
this concerns the divergence criterion for polynomials with constant coefficients
(Lemma~\ref{lem:divcrit}), which is a monomial form of the exactness
result of~\cite{HydonVariationalComplex2004}.
The proofs are presented in detail. Each of them is an ad hoc argument tailored to the specific stencil, which makes it possible to obtain a complete proof for the case at hand and to avoid cumbersome general formulations. To illustrate the ideas involved, the proof is preceded by simple examples and by the proof of a simpler model case, the $2\times2$ stencil, before it is applied to the $3\times3$ stencil.
For brevity of exposition, the variables $t$ and $x$ will sometimes be referred to as the time and space coordinates.

\subsection{Divergences and orbit sums}

Let $P$ be a polynomial in the values $u^{n+j}_{m+k}$ with constant coefficients.
Recall that $P$ is a (total difference) divergence if
$P=(\dtm{S}-1)G+(\dhm{S}-1)F$ for some polynomials $F,G$.
An equivalent description is the following (see, e.g.,~\cite{GodunovRyabenkii1987}).
Let $n_{\min}<n_{\max}$ and $m_{\min}<m_{\max}$ be integers, and let
\[
R=\{(n,m)\in\mathbb Z^2:\ n_{\min}\leqslant n\leqslant n_{\max},\ m_{\min}\leqslant m\leqslant m_{\max}\}
\]
be the rectangle of mesh nodes bounded by them (it contains all nodes between the
endpoints). Summation of $P$ over $R$ gives
\[
\sum_{(n,m)\in R}P^n_m
=\sum_{m_{\min}\leqslant m\leqslant m_{\max}}\bigl(G^{n_{\max}+1}_m-G^{n_{\min}}_m\bigr)
+\sum_{n_{\min}\leqslant n\leqslant n_{\max}}\bigl(F^{n}_{m_{\max}+1}-F^{n}_{m_{\min}}\bigr).
\]
Indeed, for fixed $m$ the sum $\sum_{n_{\min}\leqslant n\leqslant n_{\max}}(G^{n+1}_m-G^n_m)$
contains each $G^n_m$ with $n_{\min}<n\leqslant n_{\max}$ twice, once with the plus sign
and once with the minus sign, so only $G^{n_{\max}+1}_m$ and $G^{n_{\min}}_m$ remain;
the same holds for $F$ along each row.
Thus the sum of a divergence over a rectangle is determined by $G$ and $F$ at the
first and last rows and columns, and all interior terms cancel.
The converse also holds, and it is convenient to express it in terms of
monomials. For this we introduce the following notions.

\medskip

Let $N$ be a monomial in the variables $u^{n+j}_{m+k}$. Its
\emph{orbit} $\mathcal O(N)$ is the set of all monomials obtained from $N$ by a shift of
all indices by the same vector $(p,q)\in\mathbb Z^2$, i.e., by
${\dtm{S}}^p{\dhm{S}}^q$. When the monomial is not important we simply write $\mathcal O$.
For example, $u^n_m u^{n+1}_m$ and
$u^{n+4}_{m-2}u^{n+5}_{m-2}$ belong to the same orbit (two nodes adjacent in~$t$),
whereas $u^n_m u^{n}_{m+1}$ belongs to a different orbit (two nodes adjacent in~$x$).
The relation ``$N'$ is obtained from $N$ by a shift ${\dtm{S}}^p{\dhm{S}}^q$'' is
reflexive ($p=q=0$), symmetric (shift by $(-p,-q)$) and transitive (shifts compose),
hence an equivalence relation; its classes are the orbits, so that every monomial lies
in exactly one orbit.

\medskip

Let $P=\sum_N p_N N$ be a polynomial written as a linear combination
of distinct monomials $N$ with coefficients $p_N$ (we put $p_N=0$ for monomials that do not occur in $P$).
The \emph{orbit sum} of $P$ over an orbit $\mathcal O$ is the sum of the coefficients of $P$ at all
monomials of this orbit:
\[
\sigma_{\mathcal O}(P)=\sum_{N\in\mathcal O}p_N .
\]
The sum is finite, since only finitely many monomials of $\mathcal O$ occur in $P$.
For example, if $P=3\,u^n_m u^{n+1}_m-2\,u^{n+4}_{m-2}u^{n+5}_{m-2}+u^n_m u^n_{m+1}$, then the first two
monomials lie in one orbit and the third in another, so that the orbit sum over the first orbit is $3-2=1$ and
over the second it is $1$; over all other orbits it is zero.

\medskip

\begin{lemma}[Divergence criterion]\label{lem:divcrit}
A polynomial $P$ with constant coefficients is a divergence if and only if
$\sigma_{\mathcal O}(P)=0$ for every orbit $\mathcal O$.
\end{lemma}
\begin{proof}
Let $P=(\dtm{S}-1)G$ with $G=\sum g_N N$. Then
$P=\sum_N g_N(\dtm{S}N-N)$, and the two monomials $\dtm{S}N$ and $N$ lie in the same orbit
and enter with coefficients $+g_N$ and $-g_N$; the same holds for $\dhm{S}-1$.
Hence every orbit sum of a divergence vanishes.

Conversely, let all orbit sums of $P$ vanish. Only finitely many orbits contain
monomials of $P$, and only they matter; for each of them choose a representative
$N_{\mathcal O}$. Then
\[
P=\sum_{\mathcal O}\ \sum_{N\in\mathcal O}p_N\,(N-N_{\mathcal O}),
\]
where $\mathcal O$ runs over the orbits meeting $P$, since
$\sum_{N\in\mathcal O}p_N N_{\mathcal O}=\sigma_{\mathcal O}(P)\,N_{\mathcal O}=0$.
Each difference $N-N_{\mathcal O}$ is a sum of differences $N'-\dtm{S}N'$ and
$N'-\dhm{S}N'$ along a path of unit shifts from $N_{\mathcal O}$ to $N$, i.e., a divergence.
\end{proof}

\medskip

For a polynomial $\Phi$ the condition that $\Lambda$ is a conservation-law
multiplier thus reads: \emph{in $\Lambda\Phi$ all orbit sums vanish}.
Here both $\Lambda$ and $\Phi$ are polynomials with constant coefficients
(independent of $n$ and $m$), so that $\Lambda\Phi$ is again such a polynomial and the
criterion applies to it. Since $\Lambda$ is fixed, this is a system of linear equations
for the coefficients of $\Phi$, one equation for each orbit. Note that a coefficient of
$\Phi$ enters several equations: if $\Lambda$ has $r$ terms, a monomial $M$ of $\Phi$
gives $r$ monomials of $\Lambda\Phi$, which lie, as a rule, in different orbits.

For example, let $\Lambda=\Lambda_1=u^{n+1}_m-u^{n-1}_m$ and let $M=u^n_mu^n_{m+1}$ be a monomial of $\Phi$.
Then
\[
\Lambda M=u^{n+1}_m u^n_m u^n_{m+1}\;-\;u^{n-1}_m u^n_m u^n_{m+1}.
\]
The two terms lie in different orbits. Indeed, a shift moves all nodes of a monomial by the same vector,
so it preserves their positions relative to each other. In the first term the other two nodes
$u^{n+1}_m$ and $u^n_{m+1}$ are located relative to $u^n_m$ at $(+1,0)$ and $(0,+1)$ in the coordinates $(t,x)$,
in the second term the nodes $u^{n-1}_m$ and $u^n_{m+1}$ are at $(-1,0)$ and $(0,+1)$.
Hence the coefficient of $M$ enters two different equations, one for each orbit.

Conversely, different monomials of $\Phi$ may produce monomials of the same orbit:
for $\Phi=\alpha\,u^{n-1}_m+\beta\,u^{n+1}_m$ we obtain
\[
\Lambda\Phi=(\alpha-\beta)\,u^{n-1}_mu^{n+1}_m-\alpha\,(u^{n-1}_m)^2+\beta\,(u^{n+1}_m)^2 ,
\]
where $(u^{n-1}_m)^2$ and $(u^{n+1}_m)^2$ lie in the same orbit, so that its orbit sum is $\beta-\alpha$,
and the orbit sum of $u^{n-1}_mu^{n+1}_m$ is $\alpha-\beta$. Both give the single equation $\alpha=\beta$.

\medskip

Now let two multipliers $\Lambda_1$, $\Lambda_2$ be required. The polynomial $\Phi$ must then
satisfy both conditions: all orbit sums of $\Lambda_1\Phi$ vanish and all orbit sums of
$\Lambda_2\Phi$ vanish. These conditions form two systems of linear equations for the same
unknown coefficients of $\Phi$. Taken together, they define a single linear system, whose
solution space is the intersection of the solution spaces of the two systems. Each system
alone usually has many solutions, but their intersection can be much smaller, and it is this
overdetermination that limits the stencils and the degrees for which nonzero solutions exist.

A minimal example is the $2\times2$ stencil with the multipliers
$\widetilde{\Lambda}_1=u^{n+1}_m-u^n_m$ and $\widetilde{\Lambda}_2=u^n_{m+1}-u^n_m$. Denote
\[
u=u^n_m,\qquad \hat u=u^{n+1}_m,\qquad u_+=u^n_{m+1},\qquad \hat u_+=u^{n+1}_{m+1},
\]
so that $\widetilde{\Lambda}_1=\hat u-u$, $\widetilde{\Lambda}_2=u_+-u$, and take a linear
$\Phi=\alpha u+\beta\hat u+\gamma u_++\delta\hat u_+$. Then
\[
\widetilde{\Lambda}_1\Phi=(\alpha-\beta)\,u\hat u+(\beta\hat u^2-\alpha u^2)+(\delta\,\hat u\hat u_+ -\gamma\,uu_+) + \gamma\,\hat uu_+-\delta\,u\hat u_+ .
\]
The monomials of $\widetilde{\Lambda}_1\Phi$ are distributed over the orbits as follows:
\begin{center}
\def\arraystretch{1.1}
\begin{tabular}{l|l|l}
orbit & description & orbit sum \\ \hline
$\{u^2,\hat u^2\}$ & one node, shifted in $t$ & $\beta-\alpha$ \\
$\{u\hat u\}$ & two nodes adjacent in $t$ & $\alpha-\beta$ \\
$\{uu_+,\hat u\hat u_+\}$ & two nodes adjacent in $x$ & $\delta-\gamma$ \\
$\{\hat uu_+\}$ & diagonal pair & $\gamma$ \\
$\{u\hat u_+\}$ & anti-diagonal pair & $-\delta$
\end{tabular}
\end{center}
Hence $\widetilde{\Lambda}_1$ is a multiplier if and only if $\gamma=\delta=0$ and $\alpha=\beta$,
i.e., $\Phi=\alpha(u+\hat u)$; indeed $(\hat u-u)(\hat u+u)=\hat u^2-u^2=(\dtm{S}-1)u^2$.
It remains to impose the second multiplier on this family. We have
\[
\widetilde{\Lambda}_2\Phi=(u_+-u)\,\alpha(u+\hat u)=\alpha\bigl(uu_++\hat uu_+-u^2-u\hat u\bigr),
\]
and the four monomials lie in four different orbits, with orbit sums $\alpha$, $\alpha$, $-\alpha$ and $-\alpha$ for $uu_+$, $\hat uu_+$, $u^2$ and $u\hat u$, respectively. All of them must vanish, so $\alpha=0$.
Hence there is no nonzero linear $\Phi$ with both multipliers. Only constants remain~(for $\Phi=1$ the products $\widetilde{\Lambda}_i\Phi=\widetilde{\Lambda}_i$ are divergences).

\bigskip

The considered example shows how the orbit sums of two multipliers are combined, but it is restricted
to linear $\Phi$. As a more substantial, yet still simple, illustration we now take the same
$2\times2$ stencil and the same multipliers and prove that no nonconstant polynomial $\Phi$ of any degree
satisfies both conditions. The same technique is then applied to the $3\times3$ stencil.


\subsection{A model case: the $2\times2$ stencil with cell multipliers}

\bigskip

Let $u=u^n_m$, $\hat u=u^{n+1}_m$, $u_+=u^n_{m+1}$, $\hat u_+=u^{n+1}_{m+1}$
be the values at the nodes of one cell, as in the example above, and let $\Phi$ be a polynomial in $u,\hat u,u_+,\hat u_+$
with constant coefficients. Each of the two cell multipliers~(\ref{LambdasCell}),
$\widetilde\Lambda_1=\hat u-u$ and $\widetilde\Lambda_2=u_+-u$, alone admits a nonzero $\Phi$; we show
that the orbit sums of $\widetilde\Lambda_1\Phi$ and $\widetilde\Lambda_2\Phi$ cannot vanish simultaneously
unless $\Phi$ is a constant.

\begin{proposition}\label{prop:cell22}
Let $\Phi$ be a polynomial in $u,\hat u,u_+,\hat u_+$ with constant coefficients such that both
$\widetilde\Lambda_1\Phi = (\hat u-u)\Phi$ and $\widetilde\Lambda_2\Phi = (u_+-u)\Phi$ are divergences. Then $\Phi$ is a constant.
\end{proposition}

\begin{proof}
Both multipliers are homogeneous of degree~1, and the orbit sums of
Lemma~\ref{lem:divcrit} are taken separately in each degree; hence the homogeneous parts of $\Phi$
satisfy the hypotheses separately, and it suffices to show that a homogeneous $\Phi$ of degree $s\geqslant1$ vanishes.
Write
\[
\Phi=\sum_M \kappa(M)\,M,\qquad M=u^{\alpha}\hat u^{\beta}u_+^{\gamma}\hat u_+^{\delta},\quad \alpha+\beta+\gamma+\delta=s,
\]
where $M$ runs over all monomials of degree $s$ in the four variables and $\kappa(M)$ is the coefficient of $M$
(zero if $M$ does not occur in $\Phi$). Then
\[
(\hat u-u)\Phi=\sum_M\kappa(M)\,\hat uM-\sum_M\kappa(M)\,uM
=\sum_N\bigl[\kappa(N/\hat u)-\kappa(N/u)\bigr]\,N ,
\]
where $N$ runs over all monomials of degree $s+1$ in the four variables, and a term is omitted
if $N$ is not divisible by $\hat u$, respectively by $u$.
The last equality is obtained by changing the summation variable:
in the first sum we put $N=\hat uM$, i.e., $M=N/\hat u$ (only monomials $N$ divisible by $\hat u$ occur), and in the second sum
$N=uM$, i.e., $M=N/u$.
Hence the coefficient of $N$ in $(\hat u-u)\Phi$ is\footnote{To clarify the situation, consider a simple example with $s=2$ and the multiplier $\hat u-u$. Let
\[
\Phi=x_0u^2+x_1u\hat u+x_2\hat u^2,
\]
i.e., $\kappa(u^2)=x_0$, $\kappa(u\hat u)=x_1$, $\kappa(\hat u^2)=x_2$, and all other coefficients are zero.
The product $(\hat u-u)\Phi$ is a polynomial of degree~$3$, and the coefficient of each monomial $N$ is
$\kappa(N/\hat u)-\kappa(N/u)$, where a term is omitted if $N$ is not divisible by the corresponding variable
(equivalently, the corresponding $\kappa$ is zero). For the four monomials $N$ of degree~$3$ in $u,\hat u$ this gives
\begin{align*}
N=u^3:&\quad \kappa(u^3/\hat u)-\kappa(u^3/u)=0-\kappa(u^2)=-x_0,\\
N=u^2\hat u:&\quad \kappa(u^2\hat u/\hat u)-\kappa(u^2\hat u/u)=\kappa(u^2)-\kappa(u\hat u)=x_0-x_1,\\
N=u\hat u^2:&\quad \kappa(u\hat u^2/\hat u)-\kappa(u\hat u^2/u)=\kappa(u\hat u)-\kappa(\hat u^2)=x_1-x_2,\\
N=\hat u^3:&\quad \kappa(\hat u^3/\hat u)-\kappa(\hat u^3/u)=\kappa(\hat u^2)-0=x_2.
\end{align*}
Here $\kappa(u^3/\hat u)=0$ and $\kappa(\hat u^3/u)=0$, since $u^3$ is not divisible by $\hat u$ and $\hat u^3$ is not divisible by $u$.
Hence
\[
(\hat u-u)\Phi=-x_0u^3+(x_0-x_1)u^2\hat u+(x_1-x_2)u\hat u^2+x_2\hat u^3.
\]
} 
\begin{equation}\label{coefN}
\kappa(N/\hat u)-\kappa(N/u).
\end{equation}

Since all monomials of $(\hat u-u)\Phi$ are monomials in $u,\hat u,u_+,\hat u_+$, the orbit sum of $(\hat u-u)\Phi$ over the orbit of $N$ is the sum
of~(\ref{coefN}) over those shifts of $N$ that are again monomials in $u,\hat u,u_+,\hat u_+$.
By Lemma~\ref{lem:divcrit}, $(\hat u-u)\Phi$ is a divergence if and only if all these orbit sums vanish.

\smallskip

Consider orbits of monomials in $u,\hat u,u_+,\hat u_+$.
For a given $N$ we need only those of its shifts that are again monomials in $u,\hat u,u_+,\hat u_+$.
These variables are the nodes $(0,0),(1,0),(0,1),(1,1)$ of one cell (coordinates $(t,x)$),
so there are four cases.
\begin{enumerate}
\item[{(i)}]
The nodes of $N$ occupy both $t$-positions and both $x$-positions (``full''; for example, $N$ contains $u$ and $\hat u_+$, or $\hat u$ and $u_+$, or three of the variables).
Any nonzero shift moves a node out of the cell, so $N$ is the only such shift of itself and the orbit sum equals~(\ref{coefN}).

\item[{(ii)}]
$N$ contains only $u,\hat u$ or only $u_+,\hat u_+$ (one $x$-position). The shifts are $N$ and the monomial $N'$ obtained by the replacement $u\leftrightarrow u_+$, $\hat u\leftrightarrow\hat u_+$.
Exactly one of the two monomials, say $N_0$, is written in terms of $u,\hat u$; the other one contains neither $u$ nor $\hat u$, so it is divisible by neither of them, and its coefficient in $(\hat u-u)\Phi$ is zero (both terms in~(\ref{coefN}) are omitted).
Hence the orbit sum equals the coefficient of $N_0$, i.e.,
\[
\kappa(N_0/\hat u)-\kappa(N_0/u).
\]

\item[{(iii)}]
$N$ contains only $u,u_+$ or only $\hat u,\hat u_+$ (one $t$-position). The shifts are $N$ and the monomial obtained by $u\leftrightarrow\hat u$, $u_+\leftrightarrow\hat u_+$. As in case~(ii), the orbit sum is computed directly from~(\ref{coefN}); it is not needed below.

\item[{(iv)}]
$N$ is a power of a single variable; the shifts are the four powers $u^{s+1}$, $\hat u^{s+1}$, $u_+^{s+1}$, $\hat u_+^{s+1}$.
\end{enumerate}

The remaining argument consists of three steps: Step~1 shows that $\Phi$ does not contain $u_+$ and $\hat u_+$,
Step~2 determines $\Phi$ from the condition for $\widetilde\Lambda_1$, and Step~3 uses the condition for $\widetilde\Lambda_2$.
Only the cases (i) and (ii) are needed in the proof: in Step~1 the monomials $N$ are of type (i), in Step~2 of type (ii);
the cases (iii) and (iv) are listed for completeness.

\paragraph*{Step 1: $\kappa(M)=0$ if $M$ contains $u_+$ or $\hat u_+$}

Recall that $M=u^{\alpha}\hat u^{\beta}u_+^{\gamma}\hat u_+^{\delta}$ is an arbitrary monomial of degree $s$;
we show that $\kappa(M)=0$ whenever $\gamma\geqslant1$ or $\delta\geqslant1$.

(a) Let $\gamma\geqslant1$. Put $M_j=M\,(\hat u/u)^j$ for $0\leqslant j\leqslant\alpha$, i.e., $M_j$ is obtained from $M$
by replacing $j$ factors $u$ by $\hat u$; thus $M_0=M$ and $M_\alpha=\hat u^{\alpha+\beta}u_+^{\gamma}\hat u_+^{\delta}$ contains no factor $u$.
Every $M_j$ still contains $u_+$, so the monomial $N_j=\hat uM_j$ contains $\hat u$ and $u_+$;
it occupies both time levels and both space positions and falls into case~(i).
Its orbit sum is therefore~(\ref{coefN}), i.e., $\kappa(M_j)-\kappa(M_{j+1})$ (since $N_j/\hat u=M_j$ and $N_j/u=M_{j+1}$), and the equations for $N_0,\dots,N_\alpha$ form the chain
\begin{align*}
&\kappa(M_j)-\kappa(M_{j+1})=0,\qquad j=0,\dots,\alpha-1,\\
&\kappa(M_\alpha)=0 .
\end{align*}
The last equation holds because $N_\alpha$ is not divisible by $u$, so the term $\kappa(N_\alpha/u)$ is absent.
Going from the last equation upwards, we obtain $\kappa(M_\alpha)=0$, then $\kappa(M_{\alpha-1})=0$, and so on, up to $\kappa(M)=\kappa(M_0)=0$.
(For $\alpha=0$ the chain consists of the last equation only.)

(b) Let $\gamma=0$, $\delta\geqslant1$. Now the roles of $u$ and $\hat u$ are exchanged.
Put $M_j=M\,(u/\hat u)^j$ for $0\leqslant j\leqslant\beta$.
Every $N_j=uM_j$ contains $u$ and $\hat u_+$, so it falls into case~(i), and~(\ref{coefN}) gives the same chain
\begin{align*}
&\kappa(M_{j+1})-\kappa(M_j)=0,\qquad j=0,\dots,\beta-1,\\
&-\kappa(M_\beta)=0 ,
\end{align*}
the last equation because $N_\beta$ is not divisible by $\hat u$.
Hence $\kappa(M)=0$ again.

\paragraph*{Step 2: the condition for $\hat u-u$}
By Step 1, $\Phi=\sum_{i=0}^{s}\kappa_i\,u^i\hat u^{\,s-i}$ with $\kappa_i=\kappa(u^i\hat u^{\,s-i})$.
Take $N=u^i\hat u^{\,s+1-i}$ with $1\leqslant i\leqslant s$; it falls into case (ii), and its orbit contains also $N'=u_+^i\hat u_+^{\,s+1-i}$.
The coefficient~(\ref{coefN}) of $N$ is $\kappa_i-\kappa_{i-1}$, and that of $N'$ is zero, since $N'$ is divisible neither by $u$ nor by $\hat u$.
Hence the orbit sum is $\kappa_i-\kappa_{i-1}$, and its vanishing gives $\kappa_i=\kappa_{i-1}$ for $i=1,\dots,s$. Thus all coefficients $\kappa_i$ coincide; denoting their common value by $\lambda$, we obtain
\begin{equation}\label{Phi1}
\Phi=\lambda\sum_{i=0}^{s}u^i\hat u^{s-i}=\lambda\,\frac{\hat u^{s+1}-u^{s+1}}{\hat u-u}.
\end{equation}
(Conversely, $(\hat u-u)\Phi=\lambda(\hat u^{s+1}-u^{s+1})=\lambda(\dtm{S}-1)u^{s+1}$ is a divergence, but only the necessity is needed here.
For $s=1$ this is the solution $\Phi=\lambda(u+\hat u)$ of the example in the previous subsection.)

\paragraph*{Step 3: both multipliers}
The interchange $\hat u\leftrightarrow u_+$ (that is, of $t$ and $x$) maps shifts to shifts and therefore divergences to divergences, and it turns the condition for $u_+-u$
into the condition for $\hat u-u$. By Steps 1--2 it follows that $\Phi=\lambda'\sum_{i=0}^{s}u^iu_+^{\,s-i}$
for some constant $\lambda'$. The monomial $\hat u^s$ occurs in~(\ref{Phi1}) with the coefficient $\lambda$ and does not occur in this expression,
hence $\lambda=0$ and $\Phi=0$ for every $s\geqslant1$.

\end{proof}

To illustrate how Step~1 works on concrete monomials, consider the following example.
Recall that the coefficients~$\kappa(M)$ are unknown: we do not know in advance which monomials occur in $\Phi$.
Equation~(\ref{coefN}) must hold for \emph{every} monomial $N$, whether or not $N$ occurs in
$(\hat u-u)\Phi$, and we use it as a system of linear equations for the unknowns $\kappa(M)$.
Let $s=3$ and $M=u^2u_+$; we show that $\kappa(M)=0$, i.e., that $M$ does not occur in $\Phi$.
Equation~(\ref{coefN}) for the monomials
$N=u^2\hat uu_+$, $u\hat u^2u_+$ and $\hat u^3u_+$ reads
\begin{align*}
N=u^2\hat uu_+ &:\quad \kappa(u^2u_+)-\kappa(u\hat uu_+)=0,\\
N=u\hat u^2u_+ &:\quad \kappa(u\hat uu_+)-\kappa(\hat u^2u_+)=0,\\
N=\hat u^3u_+ &:\quad \kappa(\hat u^2u_+)=0 .
\end{align*}
Each of these $N$ contains both $\hat u$ and $u_+$, so it falls into case~(i) and
(\ref{coefN}) is indeed the whole orbit sum. In the last line $N$ is not divisible by $u$,
so the term $\kappa(N/u)$ is absent. Going from the last equation upwards, we obtain
$\kappa(\hat u^2u_+)=0$, then $\kappa(u\hat uu_+)=0$, and finally $\kappa(u^2u_+)=0$.
For case~(b), say $M=\hat u^2\hat u_+$, the argument is the same with $u$ and $\hat u$ exchanged:
the monomials $N=u\hat u^2\hat u_+$, $u^2\hat u\hat u_+$, $u^3\hat u_+$ give
$\kappa(\hat u^2\hat u_+)=\kappa(u\hat u\hat u_+)=\kappa(u^2\hat u_+)=0$.

\smallskip

The presence of $u_+$ or $\hat u_+$ in $M$ is essential. Consider $s=3$ and $M=u^3$, which contains neither of them.
The chain starts in the same way: the monomials $N=u^3\hat u$, $u^2\hat u^2$ and $u\hat u^3$ contain both $u$ and $\hat u$ and no other variables,
so they fall into case~(ii); the second monomial of their orbit, $N'$, is written in terms of $u_+,\hat u_+$ and has coefficient zero,
so the orbit sum is again~(\ref{coefN}), and we obtain
\begin{align*}
N=u^3\hat u &:\quad \kappa(u^3)-\kappa(u^2\hat u)=0,\\
N=u^2\hat u^2 &:\quad \kappa(u^2\hat u)-\kappa(u\hat u^2)=0,\\
N=u\hat u^3 &:\quad \kappa(u\hat u^2)-\kappa(\hat u^3)=0 .
\end{align*}
However, the next monomial of the chain is $N=\hat u^4$, a power of a single variable, so it falls into case~(iv).
It is not the only monomial of its orbit: besides $\hat u^4$, the orbit contains $u^4$ (the shift in time), $u_+^4$ and $\hat u_+^4$.
The coefficients of $\hat u^4$ and $u^4$ in $(\hat u-u)\Phi$ are $\kappa(\hat u^3)$ and $-\kappa(u^3)$, respectively,
and those of $u_+^4$ and $\hat u_+^4$ are zero, so the orbit sum gives
\[
N=\hat u^4:\quad \kappa(\hat u^3)-\kappa(u^3)=0 ,
\]
instead of an equation that sets a coefficient equal to zero.
Together with the previous three equations this gives only $\kappa(u^3)=\kappa(u^2\hat u)=\kappa(u\hat u^2)=\kappa(\hat u^3)$:
the chain closes into a cycle and does not force these coefficients to vanish.
Indeed, they do not vanish in general: $\Phi=u^3+u^2\hat u+u\hat u^2+\hat u^3$ satisfies the condition for $\hat u-u$
(see Step~2).


\subsection{The $3\times3$ stencil with centered multipliers}
\label{sec:33proof}

\smallskip

\noindent
Proposition~\ref{prop:33} is a reformulation of the following theorem.

\begin{theorem} \label{th:33all}
Let $\Phi$ be a polynomial in the nine variables of the $3\times3$ stencil with
constant coefficients such that $(u^{n+1}_m-u^{n-1}_m)\Phi$ and
$(u^{n}_{m+1}-u^{n}_{m-1})\Phi$ are total difference divergences. Then
$\deg\Phi\leqslant 1$.
\end{theorem}

The proof is divided into several relatively simple steps, with auxiliary lemmas introduced along the way.

The nine variables of the stencil are denoted
\[
\begin{array}{lll}
\check u_-=u^{n-1}_{m-1} & \check u=u^{n-1}_{m} & \check u_+=u^{n-1}_{m+1}\\
u_-=u^{n}_{m-1}   & u=u^{n}_{m}   & u_+=u^{n}_{m+1}\\
\hat u_-=u^{n+1}_{m-1} & \hat u=u^{n+1}_{m} & \hat u_+=u^{n+1}_{m+1}
\end{array}
\]
so that $\Lambda_1=\hat u-\check u$ and $\Lambda_2=u_+-u_-$.
The variables $\check u,\hat u,u_-,u_+,u$ form the \emph{cross}; the remaining four,
$\check u_\pm,\hat u_\pm$, are the \emph{corners}. The \emph{cell} of a corner is the
$2\times2$ block of the stencil containing it; for example, the cell of $\check u_-$ is
$\{\check u_-,\check u,u_-,u\}$.
For a monomial $M$ let $V(M)$ be the set of variables occurring in $M$.
A set of variables is \emph{full} if it contains a variable from the first row,
from the third row, from the first column and from the third column of the table above.
A monomial with a full set of variables cannot be shifted inside the stencil,
i.e., it is the only monomial of its orbit lying in the stencil.

As in the previous section, write $\Phi=\sum_M\kappa(M)\,M$, where $\kappa(M)$ is the
coefficient of $M$ (zero if $M$ does not occur in $\Phi$, and also if $M$ is not a monomial,
e.g.\ $M=N/\hat u$ with $\hat u\nmid N$). These coefficients are unknown;
the conditions of the theorem are a system of linear equations for them.

\paragraph*{Step 0: reduction}
Both multipliers are homogeneous of degree~1, and orbits consist of monomials of one degree.
Hence the homogeneous parts $\Phi_s$ of $\Phi$ satisfy the hypotheses separately,
and it suffices to show that $\Phi_s=0$ for every $s\geqslant2$. Fix such an $s$.

\paragraph*{Step 1: the equations}
As in the $2\times2$ case, the coefficient of a monomial $N$ of degree $s+1$ in $(\hat u-\check u)\Phi$ is
$\kappa(N/\hat u)-\kappa(N/\check u)$. Only monomials lying in the stencil occur in
$(\hat u-\check u)\Phi$, so the orbit sum over the orbit of $N$ gives, for every monomial $N$,
\begin{equation}\label{eq:E1}
\sum_{N'}\bigl[\kappa(N'/\hat u)-\kappa(N'/\check u)\bigr]=0,
\end{equation}
where $N'$ runs over all shifts of $N$ lying inside the stencil (including $N$ itself).
Likewise, for the second multiplier,
\begin{equation}\label{eq:E2}
\sum_{N'}\bigl[\kappa(N'/u_+)-\kappa(N'/u_-)\bigr]=0 .
\end{equation}
The symmetries $t\to-t$, $x\to-x$, $t\leftrightarrow x$ and the point reflection $(t,x)\to(-t,-x)$
map the pair $\{\Lambda_1,\Lambda_2\}$ into itself up to signs and interchange of its elements,
hence the image of a solution $\Phi$ under each of them is again a solution.
Consequently, a statement proved for all solutions holds also in each transformed form.
We call a monomial $N$ \emph{rigid}~(see Fig.~\ref{fig:rigid}) if no other shift of $N$ contributes to (\ref{eq:E1}), (\ref{eq:E2}),
so that for $N$ these equations read
\begin{equation}\label{eq:rigid}
\kappa(N/\hat u)=\kappa(N/\check u),\qquad \kappa(N/u_+)=\kappa(N/u_-).
\end{equation}
A monomial with a full set of variables is rigid.

\begin{figure}[ht]
\centering
\includegraphics[width=\textwidth]{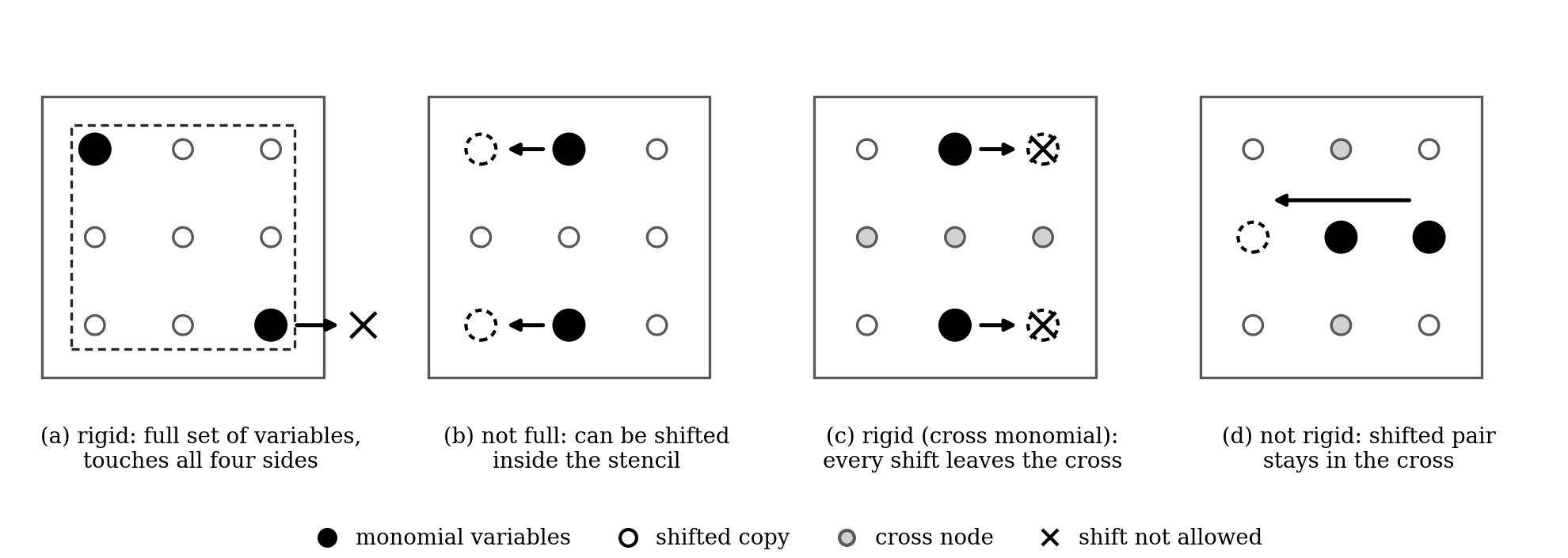}
\caption{Rigid monomials on the $3\times3$ stencil. Black circles are the variables of a monomial, open dashed circles are its shifted copies, gray nodes form the cross. (a)~A monomial with a full set of variables is rigid. (b)~A monomial that is not of this kind can be shifted inside the stencil. (c)~A monomial in the cross is rigid if every shift leaves the cross. (d)~A monomial in the cross is not rigid if some shift stays in the cross.}
\label{fig:rigid}
\end{figure}

\paragraph*{Step 2: chains}

\begin{lemma}[chain]\label{lem:chain}
Let $M$ be a monomial and let $k$ be the exponent of $\check u$ in $M$. Put $M_i=M(\hat u/\check u)^i$ for $0\leqslant i\leqslant k$.
If all monomials $N_i=\hat uM_i$ are rigid, then $\kappa(M)=0$.
The same holds with the pair $(\check u\to\hat u)$ replaced by $(\hat u\to\check u)$, $(u_-\to u_+)$ or $(u_+\to u_-)$,
with $\Lambda_1$ and $\Lambda_2$ used accordingly.
In particular, the assumption holds if $(V(M)\setminus\{\check u\})\cup\{\hat u\}$ is full,
and analogously for the other pairs.
\end{lemma}
\begin{proof}
For $i<k$ the first equation of (\ref{eq:rigid}) for $N_i$ reads $\kappa(M_i)=\kappa(M_{i+1})$,
since $N_i/\hat u=M_i$ and $N_i/\check u=M_{i+1}$.
For $i=k$ the monomial $N_k$ is not divisible by $\check u$, so the equation reads $\kappa(M_k)=0$.
Hence $\kappa(M)=\kappa(M_1)=\dots=\kappa(M_k)=0$.

For the last claim, note that every $N_i$ contains $\hat u$ and all variables of $M$ other than $\check u$, i.e.,
$V(N_i)\supset(V(M)\setminus\{\check u\})\cup\{\hat u\}$.
If the latter set is full, then so is $V(N_i)$, since adding variables to a full set keeps it full.
Hence every $N_i$ is rigid.
\end{proof}

This is the argument of Step~1 in the $2\times2$ case. An example illustrates it.
Let $s=3$ and $M=\check u_-\check u\,u_+$. Here $k=1$, $M_1=\check u_-\hat u\,u_+$, and
$(V(M)\setminus\{\check u\})\cup\{\hat u\}=\{\check u_-,\hat u,u_+\}$ is full
(it meets the first and third rows and the first and third columns).
Equation~(\ref{eq:rigid}) for $N_0=\check u_-\check u\hat u\,u_+$ and $N_1=\check u_-\hat u^2u_+$ reads
\begin{align*}
N_0:&\quad \kappa(\check u_-\check u\,u_+)-\kappa(\check u_-\hat u\,u_+)=0,\\
N_1:&\quad \kappa(\check u_-\hat u\,u_+)=0,
\end{align*}
where in the second line $N_1$ is not divisible by $\check u$. Hence $\kappa(M)=0$, i.e., $M$ does not occur in $\Phi$.

\paragraph*{Step 3: monomials containing a corner vanish}

\begin{lemma}\label{lem:corner1}
Let $\check u_-\in V(M)$ and let $V(M)$ not be contained in the cell of $\check u_-$. Then $\kappa(M)=0$.
The same holds for the other three corners and their cells.
\end{lemma}
\begin{proof}
Let $S=V(M)$. The condition $S\not\subset\{\check u_-,\check u,u_-,u\}$ means that $S$ contains
a variable of the third column ($\check u_+,u_+,\hat u_+$) or of the third row ($\hat u_-,\hat u,\hat u_+$).
In the first case, $(S\setminus\{\check u\})\cup\{\hat u\}$ contains $\check u_-$ (first row, first column), $\hat u$ (third row)
and a variable of the third column, so it is full.
In the second case, $(S\setminus\{u_-\})\cup\{u_+\}$ contains $\check u_-$, $u_+$ (third column)
and a variable of the third row, so it is full again.
By Lemma~\ref{lem:chain}, $\kappa(M)=0$.
The other corners follow by the symmetries $t\to-t$, $x\to-x$.
\end{proof}

In particular, a monomial containing two different corners vanishes, because the cell of one corner
does not contain the other. So a monomial with a corner $x$ that survives has $V(M)$ inside the cell of $x$.

\begin{lemma}\label{lem:corner2}
Every monomial containing a corner has coefficient zero.
\end{lemma}
\begin{proof}
By the symmetries $t\to-t$ and $x\to-x$ it suffices to consider
$M=\check u_-^{\alpha}\check u^{\beta}u_-^{\gamma}u^{\delta}$ with $\alpha\geqslant1$.
Take $N=\hat uM$. Its variables lie in the first two columns and occupy all three rows,
so, besides $N$, the only shift of $N$ lying in the stencil is the shift by one column,
\[
N'=\check u^{\alpha}\check u_+^{\beta}u^{\gamma}u_+^{\delta}\hat u_+ .
\]
Since $\hat u\nmid N'$, equation~(\ref{eq:E1}) reads
\[
\kappa(M)-\kappa(K_1)-\kappa(K_2)=0,\qquad
K_1=\hat u\,\check u_-^{\alpha}\check u^{\beta-1}u_-^{\gamma}u^{\delta},\quad
K_2=\check u^{\alpha-1}\check u_+^{\beta}u^{\gamma}u_+^{\delta}\hat u_+ ,
\]
where $K_1$ is absent for $\beta=0$.
The monomial $K_1$ contains $\check u_-$ and $\hat u$, which is outside the cell of $\check u_-$,
so $\kappa(K_1)=0$ by Lemma~\ref{lem:corner1}.
The monomial $K_2$ contains the corner $\hat u_+$, and also either the corner $\check u_+$ (if $\beta\geqslant1$)
or $\check u$ (if $\alpha\geqslant2$), which is outside the cell $\{\hat u_+,\hat u,u_+,u\}$ of $\hat u_+$;
in both cases $\kappa(K_2)=0$ by Lemma~\ref{lem:corner1}.
Hence $\kappa(M)=0$ if $\alpha\geqslant2$ or $\beta\geqslant1$.
Applying this to the image of $\Phi$ under $t\leftrightarrow x$
(which exchanges $\check u\leftrightarrow u_-$, $\hat u\leftrightarrow u_+$ and hence the exponents $\beta$ and $\gamma$)
we get $\kappa(M)=0$ also when $\gamma\geqslant1$.
The remaining monomials are $M=\check u_-u^{s-1}$. For them $K_1$ is absent and
$\kappa(M)=\kappa(K_2)=\kappa(u_+^{s-1}\hat u_+)$.
The point reflection $(t,x)\to(-t,-x)$ maps $u_+^{s-1}\hat u_+$ to $\check u_-u_-^{s-1}$, which has $\gamma=s-1\geqslant1$
and therefore has a zero coefficient in every solution, in particular in the image of $\Phi$.
Thus $\kappa(u_+^{s-1}\hat u_+)=0$ and $\kappa(M)=0$.
\end{proof}

\paragraph*{Step 4: monomials in the cross}
By Lemma~\ref{lem:corner2} only monomials with $V(M)\subset\{\check u,\hat u,u_-,u_+,u\}$ remain.
A shift $N'$ of a monomial $N$ contributes to (\ref{eq:E1}), (\ref{eq:E2}) only if $N'/\hat u$ or $N'/\check u$
(respectively $N'/u_\pm$) can have a nonzero coefficient.
If $N'$ contains a corner, so do these quotients, and their coefficients vanish.
Hence only shifts with $V(N')\subset$ cross matter, and a cross monomial $N$ is rigid
whenever it has no shift $N'\neq N$ in the stencil with $V(N')\subset$ cross.

\begin{lemma}\label{lem:cross}
A cross monomial $N$ is rigid if $|V(N)|\geqslant3$, or $V(N)=\{\check u,\hat u\}$, or $V(N)=\{u_-,u_+\}$.
\end{lemma}
\begin{proof}
For any nonzero shift, at most two nodes of the cross are mapped into the cross
(this is verified by a direct check of the possible shifts).
A monomial $N$ with $|V(N)|\geqslant3$ has at least three distinct variables, which would all have to be mapped into the cross,
so it cannot be shifted.
The pairs $\{\check u,\hat u\}$ and $\{u_-,u_+\}$ are the only pairs of cross nodes at distance~$2$
in the $t$- and $x$-directions, so they cannot be shifted into the cross either.
\end{proof}

\begin{lemma}\label{lem:X}
Let $M$ be a cross monomial.
(a) If $V(M)$ contains two of $u_-,u_+,u$, then $\kappa(M)=0$.
(b) If $V(M)$ contains two of $\check u,\hat u,u$, then $\kappa(M)=0$.
\end{lemma}
\begin{proof}
(a) Use the chain $\check u\to\hat u$. All $N_i=\hat uM_i$ contain $\hat u$ and the two given variables
(the replacement $\check u\to\hat u$ does not affect them), so $|V(N_i)|\geqslant3$ and $N_i$ is rigid by Lemma~\ref{lem:cross}.
(b) is the chain $u_-\to u_+$ with $N_i=u_+M_i$.
\end{proof}

After Lemma~\ref{lem:X}, the surviving cross monomials are $u^s$ and $\xi^i\eta^j$ with
$\xi\in\{\check u,\hat u\}$, $\eta\in\{u_-,u_+\}$, $i+j=s$.

\emph{(i) $i\geqslant2$ or $j\geqslant2$.}
Let $i\geqslant2$ and $\xi=\check u$. The monomial $N=\hat u\check u^{i}\eta^j$ has $V(N)=\{\check u,\hat u\}$ or $\{\check u,\hat u,\eta\}$,
so it is rigid by Lemma~\ref{lem:cross}, and (\ref{eq:E1}) gives $\kappa(\check u^i\eta^j)=\kappa(\hat u\check u^{i-1}\eta^j)$.
The right-hand side vanishes by Lemma~\ref{lem:X}(b), since $i-1\geqslant1$.
For $\xi=\hat u$ use $N=\check u\hat u^{i}\eta^j$ in the same way, and for $j\geqslant2$ use (\ref{eq:E2}) and Lemma~\ref{lem:X}(a).
In particular, all these monomials vanish for $s\geqslant3$, since then $i\geqslant2$ or $j\geqslant2$.

\emph{(ii) The monomial $u^s$.}
Take $N=\check u\,u^s$. Besides $N$, the only shift of $N$ with variables in the cross is
$N'=u\,\hat u^s$ (the shift by one step in $t$); other shifts in the stencil contain corners.
Equation~(\ref{eq:E1}) gives $-\kappa(u^s)+\kappa(u\hat u^{s-1})=0$,
and $\kappa(u\hat u^{s-1})=0$ by Lemma~\ref{lem:X}(b), because $s-1\geqslant1$.

\emph{(iii) $s=2$, $M\in\{\check u u_-,\check u u_+,\hat u u_-,\hat u u_+\}$.}
The monomials $\check u\hat uu_-$, $\check uu_-u_+$, $\hat uu_-u_+$ contain three variables and are rigid by Lemma~\ref{lem:cross}.
Equations (\ref{eq:E1}), (\ref{eq:E2}) for them read, respectively,
\[
\kappa(\check uu_-)=\kappa(\hat uu_-),\qquad
\kappa(\check uu_-)=\kappa(\check uu_+),\qquad
\kappa(\hat uu_-)=\kappa(\hat uu_+),
\]
so all four coefficients are equal. For the monomial $N=\check u^2u_-$ the shifts in the stencil with variables in the cross are $N$ and
$N'=u_+^2\hat u$ (the shift by $(1,1)$). Equation~(\ref{eq:E1}) gives $-\kappa(\check uu_-)+\kappa(u_+^2)=0$,
and $\kappa(u_+^2)=0$ by case~(i). Hence all four coefficients vanish.

\smallskip

Thus all coefficients of $\Phi_s$ vanish for every $s\geqslant2$, which proves Theorem~\ref{th:33all}.
\qed

\begin{remark}
The proof uses the small size of the stencil in an essential way. On the $3\times3$ stencil, a monomial whose variables occupy the first and the third rows and the first and the third columns cannot be shifted inside the stencil, so that it is the only monomial of its orbit lying in the stencil, and each orbit condition for it reduces to a simple relation between two coefficients. Such \emph{rigid} monomials serve as starting points of the chains of equalities (for example, in Lemma~\ref{lem:chain}), which eventually force all coefficients to vanish. On larger stencils rigid monomials still exist, but they form only a small part of all monomials: most monomials can be shifted in several ways, the orbits are long, and each orbit condition involves many coefficients at once. The chains starting from rigid monomials then no longer reach all coefficients, and nonzero solutions appear, as in Proposition~\ref{prop:55} for the $5\times5$ stencil (for the $4\times4$ stencil see, however, Remark~\ref{rem:proof4x4}). In particular, the argument above cannot be transferred to larger stencils without substantial changes.
\end{remark}

\begin{remark}
\label{rem:proof4x4}
For the $4\times4$ stencil (Proposition~\ref{prop:44}), it appears that the problem reduces to the $3\times3$ case.
With the multipliers centered at one of the four central nodes, the extra row and the extra column of the stencil
lie on one side only, and the same technique, based on shifts of monomials and chains of equalities between coefficients,
seems to show that all monomials containing variables of the extra row or of the extra column have zero coefficients.
The polynomial $\Phi$ then depends only on the variables of the $3\times3$ stencil, and Theorem~\ref{th:33all} applies,
which would extend Proposition~\ref{prop:44} to polynomials of arbitrary degree.
Preliminary results in this direction are available, but they are not included here in view of the already large volume of the paper.
\end{remark}

\bigskip

\section{Remarks on rational-function approximations}
\label{app:Rational}

As noted in Section~\ref{sec:DirectMethod}, the direct method becomes prohibitively cumbersome for rational schemes. Under certain additional assumptions, however, the problem can sometimes be simplified as follows.

\medskip

Consider the case of a single independent variable and a single dependent variable~$u$. Suppose that the divergence function is sought in the form
\[
F = \frac{P}{Q} = G^+ - G,
\]
where $G$ is an unknown potential, the denominator $Q$ is given, and $P$ is a polynomial with unknown coefficients. Let
\[
G = \frac{A}{B},
\]
where $B$ is also prescribed\footnote{Actually, if $Q$ is given, $B$ can be determined completely.} and chosen so that $Q$ divides $B B^+$, while $A$ is another polynomial with unknown coefficients. Then
\[
A^+ B - A B^+ = \frac{B B^+ P}{Q} = P R,
\]
where $R$ is a polynomial.

Thus, the problem is reduced to a polynomial identity in the variables $u_j$. Equating the coefficients of the monomials yields a linear homogeneous system for the coefficients of $A$ and $P$. Since this system is homogeneous, it always admits trivial solutions corresponding to $F=0$. Therefore, additional constraints can be imposed on the coefficients, depending on the equation being approximated. Typically, some of the coefficients are subject to linear normalization conditions of the form
\[
\sum \alpha_i = 1.
\]
Such conditions make the system inhomogeneous and exclude the trivial solution.

It is also important to take into account that the polynomials $A$ and $B$ are defined on a reduced stencil. For example, if $F$ is defined at the points $(u_m,\ldots,u_{m+k})$, then $A$ and $B$ should be sought on the stencil $(u_m,\ldots,u_{m+k-1})$.

\smallskip

Thus, if the denominator $Q$ is known in advance, the problem can be reduced to a linear one: the denominator $B$ can be determined from the irreducible factors of $Q$, while the coefficients of $A$ and $P$ are obtained from a linear system. For example, when constructing schemes for various shallow-water models in Lagrangian coordinates, assumptions about the denominator can be made based on the analysis of difference invariants or on the form of the underlying differential equation.

In addition, many physically relevant differential equations can be written in the form
\[
u_{tt} + D = 0,
\]
where $D$ is a total derivative with respect to $x$. A divergence approximation for $u_{tt}$ is constructed straightforwardly, and the problem is therefore essentially reduced to the one-dimensional case in $x$ described above. If the approximation of $D$ involves several time layers, the values on each layer can be treated as separate dependent variables; the same reasoning then applies without modification. It should be emphasized that this construction provides a sufficient, but not necessary, condition for the conservativity of the scheme, since this splitting restricts the class of schemes under consideration.

\medskip

The two-dimensional case is not considered here. Similar constructions can be obtained in this case, but the number of unknown potentials increases, and their determination is no longer unique.


\bigskip


\end{document}